\documentclass{article}
\usepackage[utf8]{inputenc}
\usepackage{amsthm,amssymb}
\usepackage{mathtools}
\usepackage{amsfonts}
\usepackage{amsmath}
\usepackage{accents}
\usepackage{stackengine}
\usepackage{graphicx}
\usepackage{authblk}
\usepackage{xcolor}
\usepackage{hyperref}
\hypersetup{hypertexnames=false,colorlinks=true,allcolors=blue}

\newcommand{\PPM}{\textsc{PPM} }
\newcommand{\CPPM}[1]{\textsc{#1-PPM}}
\newcommand{\PPPM}[1]{\textsc{#1-Pattern PPM}}

\newcommand{\cA}{\mathcal{A}}
\newcommand{\cB}{\mathcal{B}}
\newcommand{\cC}{\mathcal{C}}
\newcommand{\cD}{\mathcal{D}}
\newcommand{\cF}{\mathcal{F}}
\newcommand{\cP}{\mathcal{P}}
\newcommand{\cQ}{\mathcal{Q}}
\newcommand{\cT}{\mathcal{T}}
\newcommand{\cM}{\mathcal{M}}
\newcommand{\cN}{\mathcal{N}}

\DeclareMathOperator{\Av}{Av}

\DeclareMathOperator{\St}{St}
\DeclareMathOperator{\Grid}{Grid}

\newcommand{\Inc}{\raisebox{-0.3mm}{\includegraphics[height=0.7em]{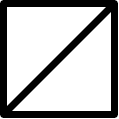}}}
\newcommand{\Dec}{\raisebox{-0.3mm}{\includegraphics[height=0.7em]{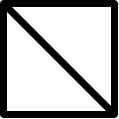}}}

\newcommand{\myarrow}{\scalebox{1.4}[1]{\color{black}$\mathclap{\curvearrowleft}\mkern2.2mu
    \mathclap{\curvearrowright}$}}
\newcommand{\flip}[1]{\begin{gathered}[b] \myarrow \\[-5pt] #1 \end{gathered}}
\newcommand{\colb}[1]{\color{blue} #1}
\newcommand{\colr}[1]{\color{red} #1}

\newcommand{\eps}{\varepsilon}

\newtheorem{proposition}{Proposition}[section]
\newtheorem{observation}[proposition]{Observation}
\newtheorem{lemma}[proposition]{Lemma}
\newtheorem{theorem}[proposition]{Theorem}
\newtheorem{corollary}[proposition]{Corollary}

\newtheorem{problem}{Open problem}

\graphicspath{{.}{./figures/}}%helpful if your graphic files are in another directory
\title{Griddings of permutations and hardness of pattern matching\thanks{Supported by the GAUK project 1766318.}}

\author[1]{Vít Jelínek\thanks{\texttt{jelinek@iuuk.mff.cuni.cz}, supported by project 18-19158S of 
the Czech Science 
Foundation.}}
\author[1]{Michal Opler\thanks{\texttt{opler@iuuk.mff.cuni.cz}, supported by project 21-32817S of 
the Czech Science Foundation.}}
\author[2]{Jakub Pekárek\thanks{\texttt{pekarej@kam.mff.cuni.cz}}}

\affil[1]{Computer Science Institute, Charles University, Prague, Czechia}
\affil[2]{Department of Applied Mathematics, Charles University, Prague, Czechia}

\begin{document}

\maketitle

\begin{abstract}
We study the complexity of the decision problem known as \emph{Permutation Pattern Matching}, 
or~PPM. The input of PPM consists of a pair of permutations $\tau$ (the `text') and $\pi$ (the 
`pattern'), and the goal is to decide whether $\tau$ contains $\pi$ as a subpermutation. On general 
inputs, PPM is known to be NP-complete by a result of Bose, Buss and Lubiw. In this paper, we focus 
on restricted instances of PPM where the text is assumed to avoid a fixed (small) pattern~$\sigma$; 
this restriction is known as $\Av(\sigma)$-PPM. It has been previously shown that $\Av(\sigma)$-PPM 
is polynomial for any $\sigma$ of size at most 3, while it is NP-hard for any $\sigma$ containing a 
monotone subsequence of length four. 

In this paper, we present a new hardness reduction which allows us to show, in a uniform way, that 
$\Av(\sigma)$-PPM is hard for every $\sigma$ of size at least 6, for every $\sigma$ of size 5 
except the symmetry class of $41352$, as well as for every $\sigma$ symmetric to one of the three 
permutations $4321$, $4312$ and~$4231$. Moreover, assuming the exponential time hypothesis, none of 
these hard cases of $\Av(\sigma)$-PPM can be solved in time $2^{o(n/\log n)}$. Previously, such 
conditional lower bound was not known even for the unconstrained PPM problem.

On the tractability side, we combine the CSP approach of Guillemot and Marx with the structural 
results of Huczynska and Vatter to show that for any monotone-griddable permutation class $\cC$, 
PPM is polynomial when the text is restricted to a permutation from~$\cC$.
\end{abstract}

\section{Introduction}

\emph{Permutation Pattern Matching}, or PPM, is one of the most fundamental decision problems 
related to permutations. In PPM, the input consists of two permutations: $\tau$, referred to as the 
`text', and $\pi$, referred to as the `pattern'. The two permutations are represented as sequences 
of distinct integers. The goal is to determine whether the text $\tau$ contains the pattern 
$\sigma$, that is, whether $\tau$ has a subsequence order-isomorphic to $\sigma$ (see 
Section~\ref{sec-prelims} for precise definitions). 

Bose, Buss and Lubiw~\cite{BBL} have shown that the \PPM problem is NP-complete. Thus, most recent 
research into the complexity of PPM focuses either on para\-met\-rized or superpolynomial 
algorithms, or 
on restricted instances of the problem.

For a pattern $\pi$ of size $k$ and a text $\tau$ of size $n$, a straightforward brute-force approach 
can solve \PPM in time $O(n^{k+1})$. This was improved by Ahal and Rabinovich~\cite{AR08_subpattern} 
to $O(n^{0.47k+o(k)})$, and then by Berendsohn, Kozma and Marx~\cite{Berendsohn19} to $O(n^{k/4})$. 

When $k$ is large in terms of $n$, a brute-force approach solves \PPM in time $O(2^{n+o(n)})$. The 
first improvement upon this bound was obtained by Bruner and Lackner~\cite{BL_runs}, whose algorithm 
achieves the running time $O(1.79^n)$, which was in turn improved by Berendsohn, Kozma and 
Marx~\cite{Berendsohn19} to $O(1.6181^n)$.

Guillemot and Marx~\cite{Guillemot2014} have shown, perhaps surprisingly, that PPM is 
fixed-parameter tractable with parameter $k$, via an algorithm with running time $n\cdot 
2^{O(k^2\log k)}$, later improved to $n\cdot 2^{O(k^2)}$ by Fox~\cite{Fox}. 

\paragraph*{Restricted instances} 
Given that \PPM is NP-hard on general inputs,  various authors have sought to identify restrictions 
on the input permutations that would allow for an efficient pattern matching algorithm. These 
restrictions usually take the form of specifying that the pattern must belong to a prescribed set 
$\cC$ of permutations (the so-called \PPPM{$\cC$} problem), or that both the pattern and the text 
must belong to a set $\cC$ (known as \CPPM{$\cC$} problem). The most commonly considered choices for 
$\cC$ are sets of the form $\Av(\sigma)$ of all the permutations that do not contain a fixed pattern 
$\sigma$. 

Note that for the class $\Av(21)$, consisting of all the increasing permutations, \PPPM{$\Av(21)$} 
corresponds to the problem of finding the longest increasing subsequence in the given text, a well-known polynomially solvable problem~\cite{Schensted}. Another polynomially solvable case is 
\PPPM{$\Av(132)$}, which follows from more general results of Bose et al.~\cite{BBL}.

In contrast, for the class $\Av(321)$ of permutations avoiding a decreasing subsequence of length 3 
(or equivalently, the class of permutations formed by merging two increasing sequences), 
\PPPM{$\Av(321)$} is already NP-complete, as shown by Jelínek and Kynčl~\cite{JeKy}. In fact, 
Jelínek and Kynčl show that \PPPM{$\Av(\sigma)$} is polynomial for 
$\sigma\in\{1,12,21,132,231,312,213\}$ and NP-complete otherwise.

For the more restricted \CPPM{$\Av(\sigma)$} problem, a polynomial algorithm for 
$\sigma=321$ was found by Guillemot and Vialette~\cite{GV09_321} (see also Albert et 
al.~\cite{ALLV}), and it follows that \CPPM{$\Av(\sigma)$} is polynomial for any 
$\sigma$ of length at most~3. In contrast, the case $\sigma=4321$ (and by symmetry also 
$\sigma=1234$) is NP-complete~\cite{JeKy}. It follows that  \CPPM{$\Av(\sigma)$} is NP-complete 
whenever $\sigma$ contains $1234$ or $4321$ as subpermutation, and in particular, it is NP-complete 
for any $\sigma$ of length 10 or more.

In this paper, our main motivation is to close the gap between the polynomial and the NP-complete 
cases of \CPPM{$\Av(\sigma)$}. We develop a general type of hardness reduction, applicable to any 
permutation class that contains a suitable grid-like substructure. We then verify that for most 
choices of $\sigma$ large enough, the class $\Av(\sigma)$ contains the required substructure.
Specifically, we can prove that \CPPM{$\Av(\sigma)$} is NP-complete in the following cases:
\begin{itemize}
\item Any $\sigma$ of size at least 6.
\item Any $\sigma$ of size 5, except the symmetry type of $41352$ (i.e., the two symmetric 
permutations $41352$ and  $25314$).
\item Any $\sigma$ symmetric to one of $4321$, $4312$ or~$4231$.
\end{itemize}
Note that the list above includes the previously known case $\sigma=4321$. Our hardness reduction, 
apart from being more general than previous results, has also the advantage of being more efficient: 
we reduce an instance of 3-SAT of size $m$ to an instance of \PPM of size $O(m\log m)$. This implies, 
assuming the exponential time hypothesis (ETH), that none of these NP-complete cases of 
\CPPM{$\Av(\sigma)$} can be solved in time $2^{o(n/\log n)}$. Previously, this lower bound was not 
known to hold even for the unconstrained PPM problem.

\paragraph*{Grid classes}
The sets of permutations of the form $\Av(\sigma)$, i.e., the sets determined by a single forbidden 
pattern, are the most common type of permutation sets considered; however, such sets are not 
necessarily the most convenient tools to understand the precise boundary between polynomial 
and NP-complete cases of PPM. We will instead work with the more general concept of 
\emph{permutation class}, which is a set $\cC$ of permutations with the property that for any 
$\pi\in\cC$, all the subpermutations of $\pi$ are in $\cC$ as well. 

A particularly useful family of permutation classes are the so-called grid classes. When dealing 
with grid classes, it is useful to represent a permutation $\pi=\pi_1\pi_2\dotsb\pi_n$ by its 
\emph{diagram}, which is the set of points $\{(i,\pi_i)\;|\; i=1,\dotsc,n\}$. A grid class is defined 
in terms of a \emph{gridding matrix} $\cM$, whose entries are (possibly empty) permutation classes. 
We say that a permutation $\pi$ has an \emph{$\cM$-gridding}, if its diagram can be 
partitioned, by horizontal and vertical cuts, into an array of rectangles, where each rectangle 
induces in $\pi$ a subpermutation from the permutation class in the corresponding cell of~$\cM$. The 
permutation class $\Grid(\cM)$ then consists of all the permutations that have an $\cM$-gridding. 

To a gridding matrix $\cM$ we associate a \emph{cell graph}, which is the graph whose vertices are 
the entries in $\cM$ that correspond to infinite classes, with two vertices being adjacent if they 
belong to the same row or column of $\cM$ and there is no other infinite entry of $\cM$ between them.

In the griddings we consider in this paper, a prominent role is played by four specific 
classes, forming two symmetry pairs: one pair are the monotone classes $\Av(21)$ and $\Av(12)$, 
containing all the increasing and all the decreasing permutations, respectively. Note that any 
infinite class of permutations contains at least one of $\Av(12)$ and $\Av(12)$ as a subclass, by the 
Erd\H os--Szekeres theorem~\cite{ESz}. 

The other relevant pair of classes involves the so-called \emph{Fibonacci class}, denoted $\oplus21$, 
and its mirror image $\ominus12$. The Fibonacci class can be  defined as the class of permutations 
avoiding the three patterns $321$, $312$ and $231$, or equivalently, it is the class of 
permutations $\pi=\pi_1\pi_2\dotsb\pi_n$ satisfying $|\pi_i-i|\le 1$ for every~$i$. 

Griddings have been previously used, sometimes implicitly, in the analysis of special cases of PPM, 
where they were applied both in the design of polynomial algorithms~\cite{ALLV,GV09_321}, and in 
NP-hardness proofs~\cite{JeKy,Jelinek2020}. In fact, all the known NP-hardness arguments for special 
cases of \PPPM{$\cC$} are based on the existence of suitable grid subclasses of the class~$\cC$. In 
particular, previous results of the authors~\cite{Jelinek2020} imply that for any gridding matrix 
$\cM$ that only involves monotone or Fibonacci cells, \PPPM{$\Grid(\cM)$} is polynomial when the cell 
graph of $\cM$ is a forest, and it is NP-complete otherwise. Of course, if \PPPM{$\Grid(\cM)$} is 
polynomial then \CPPM{$\Grid(\cM)$} is polynomial as well. However, the results in this paper 
identify a broad family of examples where \CPPM{$\Grid(\cM)$} is polynomial, while 
\PPPM{$\Grid(\cM)$} is known to be NP-complete.

Our main hardness result, Theorem~\ref{thm-hardness}, can be informally rephrased as a claim that 
\CPPM{$\cC$} is hard for a class $\cC$ whenever $\cC$ contains, for each $n$ and a fixed $\eps>0$, a 
grid subclass whose cell graph is a path of length $n$, and at least $\eps n$ of its cells are 
Fibonacci classes. A somewhat less technical consequence, Corollary~\ref{cor:cycle-hard}, says that 
\CPPM{$\Grid(\cM)$} is NP-hard whenever the cell graph of $\cM$ is a cycle with no three vertices in 
the same row or column and with at least one Fibonacci cell.

Corollary~\ref{cor:cycle-hard} is, in a certain sense, best possible, since our main tractability 
result, Theorem~\ref{thm-monotone}, states that \CPPM{$\cC$} is polynomial whenever $\cC$ is 
\emph{monotone-griddable}, that is, $\cC\subseteq\Grid(\cM)$, where $\cM$ contains only monotone (or 
empty) cells. Moreover, by a result of Huczynska and Vatter~\cite{Huczynska}, every class $\cC$ that 
does not contain $\oplus 21$ or $\ominus 12$ is monotone griddable. Taken together, these results 
show that \CPPM{$\Grid(\cM)$} is polynomial whenever no cell of $\cM$ contains $\oplus 21$ or 
$\ominus 12$ as a subclass.

\section{Preliminaries}\label{sec-prelims}
A \emph{permutation of length $n$} is a sequence $\pi_1, \dotsc, \pi_n$ in which each element of the set $[n] = \lbrace 1, 2, \dots, n\rbrace$ appears exactly once. When writing out short permutations explicitly, we shall 
omit all punctuation and write, e.g., $15342$ for the permutation $1,5,3,4,2$.  The \emph{permutation diagram} of $\pi$ is the set of points $S_\pi = \{(i,\pi_i)\;|\;i\in[n]\}$ in the plane. Observe that no two points from $S_\pi$ share the same $x$- or $y$-coordinate. We say that such a set is in \emph{general position}. Note that we blur the distinction between permutations and their permutation diagrams, e.g., we shall refer to `the point of $\pi$'.

For a point $p$ in the plane, we let $p.x$ denote its horizontal coordinate and $p.y$ its vertical coordinate. Two finite sets $S, R \subseteq \mathbb{R}^2$ in general position are \emph{order-isomorphic}, or just \emph{isomorphic} for short, if there is a bijection $f\colon S \to R$ such that for any pair of points $p \neq q$ of $R$ we have $f(p).x < f(q).x$ if and only if $p.x < q.x$, and $f(p).y < f(p).y$ if and only if $p.y < q.y$; in such case, the function $f$ is the \emph{isomorphism} from $S$ to~$R$. The \emph{reduction} of a finite set $S \subseteq \mathbb{R}^2$ in general position is the unique permutation $\pi$ such that $S$ is isomorphic to $S_\pi$.

% We write $\pi = \red(S)$. Not used anywhere else in the text.

A permutation $\tau$ \emph{contains} a permutation $\pi$, denoted by $\pi \preceq \tau$, if there is a subset $P \subseteq S_\tau$ that is isomorphic to $S_\pi$. Such a subset is then called \emph{an occurrence} of $\pi$ in $\tau$, and the isomorphism from $S$ to $P$ is an \emph{embedding} of $\pi$ into~$\tau$. If $\tau$ does not contain $\pi$, we say that $\tau$ \emph{avoids}~$\pi$.

A \emph{permutation class} is any down-set $\cC$ of permutations, i.e., a set $\cC$ such that if $\pi \in \cC$ and $\sigma \preceq \pi$  then also $\sigma \in \cC$. For a permutation 
$\sigma$, we let $\Av(\sigma)$ denote the class of all $\sigma$-avoiding permutations. We shall 
throughout use the symbols $\Inc$ and $\Dec$ as short-hands 
for the class of increasing permutations $\Av(21)$ and the class of decreasing permutations 
$\Av(12)$.

Observe that for every permutation $\pi$ of length at most $m$, the permutation diagram $S_\pi$ is a subset of the set $\{p \mid \frac{1}{2} < p.x < m +\frac{1}{2} \wedge \frac{1}{2} < p.y < m + \frac{1}{2}\}$, called \emph{$m$-box}.  This fact motivates us to extend the usual permutation symmetries to bijections of the whole $m$-box. In particular, there are eight symmetries generated by:
\begin{description}
	\item[reversal] which reflects the $m$-box through its vertical axis, i.e., the image of a point $p$ is the point $(m + 1 - p.x, p.y)$,
	\item[complement] which reflects the $m$-box through its horizontal axis, i.e., the image of a point $p$ is the point $(p.x, m + 1 - p.y)$,
	\item[inverse] which reflects the $m$-box through its ascending diagonal axis, i.e., the image of a point $p$ is the point $(p.y, p.x)$.
\end{description}

We say that a permutation $\pi$ is \emph{symmetric} to a permutation $\sigma$ if $\pi$ can be 
transformed into $\sigma$ by any of the eight symmetries generated by reversal, complement and 
inverse. The \emph{symmetry type}\footnote{We chose the term `symmetry type' over the more 
customary `symmetry class', to avoid possible confusion with the notion of permutation class.} of a 
permutation $\sigma$ is the set of all the permutations symmetric to~$\sigma$.

The symmetries generated by reversal, complement and inverse can be applied not only to individual 
permutations but also to their classes. Formally, if $\Psi$ is one of the eight symmetries and $\cC$ 
is a permutation class, then $\Psi(\cC)$ refers to the set $\{\Psi(\sigma)|\; \sigma\in\cC\}$.  We 
may easily see that $\Psi(\cC)$ is again a permutation class. 

Consider a pair of permutations $\pi$ of length $n$ and $\sigma$ of length $m$. 
The \emph{inflation} of a point $p$ of $\pi$ by $\sigma$ is the reduction of the point set
\[S_\pi \setminus \{p\} \cup \left\{\left(p.x + \frac{q.x}{m+1}, p.y+\frac{q.y}{m+1}\right) \;\middle|\; q \in S_\sigma\right\}.\]
Informally, we replace the point $p$ with a tiny copy of $\sigma$.

The \emph{direct sum} of $\pi$ and $\sigma$, denoted by $\pi \oplus \sigma$, is the result of inflating the `1' in 12 with $\pi$ and then inflating the `2' with $\sigma$. Similarly, the \emph{skew sum} of $\pi$ and $\sigma$, denoted by $\pi \ominus \sigma$, is the result of inflating the `2' in 21 with $\pi$ and then inflating the `1' with $\sigma$. If a permutation $\tau$ cannot be obtained as direct sum of two shorter permutations, we say that $\tau$ is \emph{sum-indecomposable} and if it cannot be obtained as a skew sum of two shorter permutations, we say that it is \emph{skew-indecomposable}. Moreover, we say that a permutation class $\cC$ is \emph{sum-closed} if for any $\pi, \sigma \in \cC$ we have $\pi \oplus \sigma \in \cC$. We define \emph{skew-closed} analogously.

We define the \emph{sum completion} of a permutation $\pi$ to be the permutation class
\[\oplus \pi = \{\sigma_1 \oplus \sigma_2\oplus \cdots \oplus \sigma_k \mid \sigma_i \preceq \pi \text{ for all } i \le k \in \mathbb{N}\}.\]
Analogously, we define the \emph{skew completion} $\ominus \pi$ of $\pi$. The class $\oplus21$ is known as the \emph{Fibonacci class}.

\subsection{Grid classes}
When we deal with matrices, we number their rows from bottom to top to be consistent with the Cartesian coordinates we use for permutation diagrams. For the same reason, we let the column coordinates precede the row coordinates; in particular,  a \emph{$k\times\ell$ matrix} is a matrix with $k$ columns and $\ell$ rows, and for a matrix $\cM$, we let $\cM_{i,j}$ denote its entry in column $i$ and row~$j$. 

A matrix $\cM$ whose entries are permutation classes is called a \emph{gridding 
  matrix}. Moreover, if the entries of $\cM$ belong to the set $\{\Inc,\Dec,\emptyset\}$ then we say 
that $\cM$ is a \emph{monotone gridding matrix}. 

A \emph{$k \times \ell$-gridding} of a permutation 
$\pi$ of length $n$ are two weakly increasing sequences $1 = c_1 \le \cdots \le c_{k+1} = n+1$ and $1 = r_1 \le \cdots \le r_{\ell+1} = n + 1$.
For each $i \in [k]$ and $j \in [\ell]$, we call the set of points $p \in S_\pi$ such that $c_i \le p.x < c_{i+1}$ and $r_j \le p.y < r_{j+1}$ the \emph{$(i,j)$-cell} of $\pi$.
An \emph{$\cM$-gridding} of a permutation $\pi$ is a $k \times 
\ell$-gridding such that the reduction of the $(i,j)$-cell of $\pi$ belongs to the class~$\cM_{i,j}$ for every $i \in [k]$ and $j \in [\ell]$. If $\pi$ has an $\cM$-gridding, then $\pi$ is said 
to be \emph{$\cM$-griddable}, and the \emph{grid class of $\cM$}, denoted by $\Grid(\cM)$, is the class of all $\cM$-griddable permutations.

The \emph{cell graph} of the gridding matrix $\cM$, denoted $G_\cM$, is the graph whose vertices 
are pairs $(i,j)$ such that $\cM_{i,j}$ is an infinite class, with two vertices being adjacent if they 
share a row or a column of $\cM$ and all the entries between them are finite or empty. See Figure~\ref{fig:grid-class}. We slightly abuse the notation and use the vertices of $G_\cM$ for indexing $\cM$, i.e., for a vertex $v$ of $G_\cM$, we write $\cM_v$ to denote the corresponding entry.

A \emph{proper-turning path} in $G_\cM$ is a path $P$ such that no three vertices of $P$ share the same 
row or column. Similarly, a \emph{proper-turning cycle} in $G_\cM$ is a cycle $C$ such that no three vertices of $C$ share the same row or column.

\begin{figure}
  \centering
  \raisebox{-0.5\height}{\includegraphics[width=0.4\textwidth]{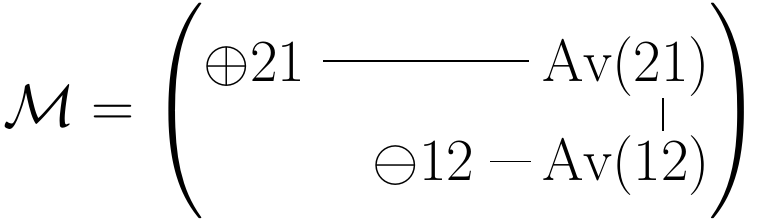}}
  \hspace{0.5in}
  \raisebox{-0.5\height}{\includegraphics[scale=0.5]{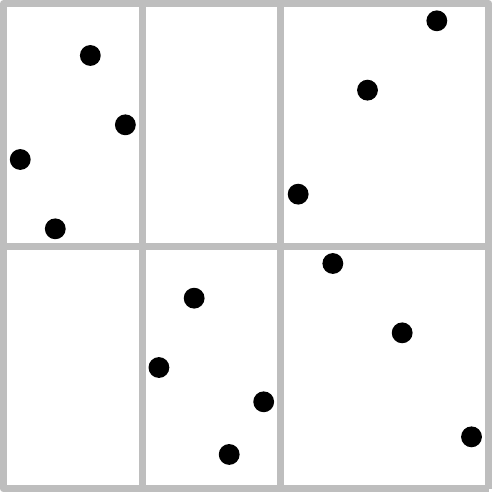}}
  \caption{A gridding matrix $\cM$ on the left and a permutation equipped with an $\cM$-gridding on the right. Empty 
    entries of $\cM$ are omitted and the edges of $G_\cM$ are displayed inside $\cM$.}
  \label{fig:grid-class}
\end{figure}

Let $\pi$ be a permutation, and let $(c,r)$ be its $k\times\ell$-gridding, where 
$c=(c_1,\dotsc,c_{k+1})$ and $r=(r_1,\dotsc,r_{\ell+1})$. A permutation $\pi$ together with a 
gridding $(c,r)$ form a \emph{gridded permutation}. When dealing with gridded permutations, it is 
often convenient to apply symmetry transforms to individual columns or rows of the gridding. 
Specifically, the \emph{reversal of the $i$-th column} of $\pi$ is the operation which generates a 
new $(c,r)$-gridded permutation $\pi'$ by taking the diagram of $\pi$, and then reflecting the 
rectangle $[c_i,c_{i+1}-1]\times[1,n]$ in the diagram through its vertical axis, producing the 
diagram of the new permutation~$\pi'$. Note that $\pi'$ differs from $\pi$ by having all the entries 
at positions $c_i, c_i+1,\dotsc,c_{i+1}-1$ in reverse order.  If $c_{i+1}\le c_i+1$, then $\pi'=\pi$.

Similarly, the \emph{complementation of the $j$-th row} of the $(c,r)$-gridded permutation $\pi$ is 
obtained by taking the rectangle $[1,n]\times[r_j,r_{j+1}-1]$ and turning it upside down, obtaining 
a permutation diagram of a new permutation.

Column reversals and row complementations can also be applied to gridding matrices: a reversal of a 
column $i$ in a gridding matrix $\cM$ simply replaces all the classes appearing in the entries of 
the $i$-th column by their reverses; a row complementation is defined analogously. 

We often need to perform several column reversals and row complementations at once. To describe such 
operations succinctly, we introduce the concept of $k\times\ell$-orientation. A 
\emph{$k\times\ell$-orientation} is a pair of functions $\cF=(f_c,f_r)$ with $f_c\colon[k]\to 
\{-1,1\}$ and $f_r\colon[\ell]\to\{-1,1\}$. To \emph{apply} the orientation $\cF$ to a 
$k\times\ell$-gridded permutation $\pi$ means to create a new permutation $\cF(\pi)$ by reversing in 
$\pi$ each column $i$ for which $f_c(i)=-1$ and complementing each row $j$ for which $f_r(j)=-1$. 
Note that the order in which we perform the reversals and complementations does not affect the final 
outcome. Note also that $\cF$ is an involution, that is, $\cF(\cF(\pi))=\pi$ for any 
$k\times\ell$-gridded permutation~$\pi$. 

We may again also apply $\cF$ to a gridding matrix $\cM$. By performing, in some order, the row 
reversals and column complementations prescribed by $\cF$ on the matrix $\cM$, we obtain a new 
gridding matrix $\cF(\cM)$. For instance, taking the gridding matrix $\left(\begin{smallmatrix} \Dec 
&\Inc\\ \Inc&\Dec\end{smallmatrix}\right)$ and applying reversal to its first column yields the 
gridding matrix $\left(\begin{smallmatrix} \Inc &\Inc\\ \Dec&\Dec\end{smallmatrix}\right)$. 
Observe that if $(c,r)$ is an $\cM$-gridding of a permutation $\pi$, then the same gridding $(c,r)$ 
is also an $\cF(\cM)$-gridding of the permutation~$\cF(\pi)$. 

Let $\cM$ be a monotone gridding matrix. An orientation $\cF$ of $\cM$ is \emph{consistent} if 
all the nonempty entries  of $\cF(\cM)$ are equal to~$\Inc$. For instance, the matrix 
$\left(\begin{smallmatrix} \Dec &\Inc\\ \Inc&\Dec\end{smallmatrix}\right)$ has a consistent 
orientation acting by reversing the first column and complementing the first row, while the matrix 
$\left(\begin{smallmatrix} \Inc &\Inc\\ \Inc&\Dec\end{smallmatrix}\right)$ has no consistent 
orientation.
% We remark that  Vatter and Waton~\cite{Vatter2011} have shown that any monotone 
% gridding matrix whose cell graph is acyclic has a consistent orientation.

A vital role in our arguments is played by the concept of monotone griddability. We say that a class 
$\cC$ is \emph{monotone-griddable} if there exists a monotone gridding matrix $\cM$ such that $\cC$ 
is contained in $\Grid(\cM)$. Huczynska and Vatter~\cite{Huczynska} provided a neat and useful 
characterization of monotone-griddable classes.

\begin{theorem}[Huczynska and Vatter~\cite{Huczynska}]
\label{thm-monotone-griddable}
A permutation class $\cC$ is monotone-griddable if and only if it contains neither the Fibonacci class $\oplus21$ nor its symmetry $\ominus12$.
\end{theorem}

A monotone grid class $\Grid(\cC\;\cD)$ where both $\cC$ and $\cD$ are non-empty is called a 
\emph{horizontal monotone juxtaposition}. Analogously, a \emph{vertical monotone juxtaposition} is a 
monotone grid class $\Grid\left(\begin{smallmatrix}\cC \\ \cD\end{smallmatrix}\right)$ with both 
$\cC$ and $\cD$ non-empty. A \emph{monotone juxtaposition} is simply a class that is either a 
horizontal or a vertical monotone juxtaposition.

\subsection{Pattern matching complexity}

In this paper, we deal with the complexity of the decision problem known as 
\CPPM{$\cC$}. For a permutation class $\cC$, the input of \CPPM{$\cC$} is a pair of permutations 
$(\pi,\tau)$ with both $\pi$ and $\tau$ belonging to $\cC$. An instance of \CPPM{$\cC$} is then 
accepted if $\tau$ contains $\pi$, and rejected if $\tau$ avoids~$\pi$. In the context of 
pattern-matching, $\pi$ is referred to as the \emph{pattern}, while $\tau$ is the \emph{text}.

Note that an algorithm for \CPPM{$\cC$} does not need to verify that the two input permutations 
belong to the class~$\cC$, and the algorithm may answer arbitrarily on inputs that fail to fulfill 
this constraint. Decision problems that place this sort of validity restrictions on their inputs are 
commonly known as \emph{promise problems}. 

Our NP-hardness results for \CPPM{$\cC$} are based on a general reduction scheme from the classical 
3-SAT problem. Given that \CPPM{$\cC$} is a promise problem, the reduction must map instances of 
3-SAT to valid instances of \CPPM{$\cC$}, i.e., the instances where both $\pi$ and $\tau$ belong 
to~$\cC$.

On top of NP-hardness arguments, we also provide time-complexity lower bounds for the hard cases
of  \CPPM{$\cC$}. These lower bounds are conditioned on the \emph{exponential-time hypothesis} 
(ETH), a classical hardness assumption which states that there is a constant $\eps>0$ such 
that 3-SAT cannot be solved in time $O(2^{\eps n})$, where $n$ is the number of variables of the 
3-SAT instance. In particular, ETH implies that 3-SAT cannot be solved in time~$2^{o(n)}$.

Given an instance $(\pi,\tau)$ of \CPPM{$\cC$}, we always use $n$ to denote the length of the 
text~$\tau$. We also freely assume that $\pi$ has length at most $n$ since otherwise the instance 
can be straightforwardly rejected. Following established practice, we express our complexity bounds 
for \CPPM{$\cC$} in terms of~$n$. Note that inputs of \CPPM{$\cC$} of size $n$ actually require 
$\Theta(n\log n)$ bits to encode.

\section{Hardness of PPM}

In this section, we present the main technical hardness result and then derive its several 
corollaries. However, we first need to introduce one more definition. 

We say that a permutation class $\cC$ has the \emph{$\cD$-rich path property} for a class $\cD$ if there is a positive constant $\eps$ such that for every $k$, the class $\cC$ contains a grid subclass whose cell graph is a proper-turning path of length $k$ with at least $\eps \cdot k$ entries equal to $\cD$.
Moreover, we say that $\cC$ has the \emph{computable $\cD$-rich path property}, if $\cC$ has the $\cD$-rich path property and there is an algorithm that, for a given $k$, outputs a witnessing proper-turning path of length $k$ with at least $\eps \cdot k$ copies of $\cD$ in time polynomial in~$k$.

\begin{theorem}
\label{thm-hardness}
Let $\cC$ be a permutation class with the computable $\cD$-rich path property for a 
non-monotone-griddable class $\cD$. Then \CPPM{$\cC$} is NP-complete, and unless ETH fails, there 
can be no algorithm that solves \CPPM{$\cC$}
\begin{itemize}%[noitemsep]
 \item in time $2^{o(n /\log{n})}$ if $\cD$ moreover contains any monotone juxtaposition,
 \item in time $2^{o(\sqrt{n})}$ otherwise.
\end{itemize}
\end{theorem}

We remark, without going into detail, that the two lower bounds we obtained under ETH are close to 
optimal. It is clear that the bound of $2^{o(n/\log{n})}$ matches, up to the $\log n$ term in the 
exponent, the trivial $2^{O(n)}$ brute-force algorithm for PPM. Moreover, the lower bound of  
$2^{o(\sqrt{n})}$ for \CPPM{$\cC$} also cannot be substantially improved without adding assumptions 
about the class~$\cC$. Consider for instance the class $\cC=\left(\begin{smallmatrix}\Dec &\Inc\\ 
\oplus21&\Dec \end{smallmatrix}\right)$. As we shall see in Proposition~\ref{pro-cyclehard}, this 
class has the computable $\oplus21$-rich path property, and therefore the $2^{o(\sqrt{n})}$ 
conditional lower bound applies to it. However, by using the technique of Ahal and 
Rabinovich~\cite{AR08_subpattern}, which is based on the concept of treewidth of permutations, we can 
solve \CPPM{$\cC$} (even \PPPM{$\cC$}) in time $n^{O(\sqrt{n})}$. This is because we can show that a 
permutation $\pi\in\cC$ of size $n$ has treewidth at most $O(\sqrt{n})$. We omit the details of the 
argument here.

\subsection{Overview of the proof of Theorem~\ref{thm-hardness}}
The proof of Theorem~\ref{thm-hardness} is based on a reduction from the well-known 3-SAT problem. 
The individual steps of the construction are rather technical, and we therefore begin with a general overview 
if the reduction. In Subsection~\ref{ssec-reduction}, we then present the reduction in full detail, together 
with the proof of correctness.

Suppose that $\cC$ is a class with the computable $\cD$-rich path property, where $\cD$ is not 
monotone griddable. This means that $\cD$ contains the Fibonacci class $\oplus21$ or its reversal 
$\ominus12$ as subclass. Suppose then, without loss of generality, that $\cD$ contains~$\oplus 21$.

To reduce 3-SAT to \CPPM{$\cC$}, consider a 3-SAT formula $\Phi$, with $n$ variables $x_1,\dotsc,x_n$ 
and $m$ clauses. We may assume that each clause of $\Phi$ has exactly 3 literals. 

Let $L=L(m,n)$ be an integer whose value will be specified later. By the $\cD$-rich path property, 
$\cC$ contains a grid subclass $\Grid(\cM)$ where the cell graph of $\cM$ is a path of length 
$L$, in which a constant fraction of cells is equal to $\cD$. 

To simplify our notation in this high-level overview, we will assume that the cell graph 
of $\cM$ corresponds to an increasing staircase. More precisely, the cells of $\cM$ representing 
infinite classes can be arranged into a sequence $C_1,C_2,\dotsc,C_L$, where $C_1$ is the 
bottom-left cell $\cM_{1,1}$ of $\cM$, each odd-numbered cell $C_{2i-1}$ 
corresponds to the diagonal cell $\cM_{i,i}$, and each even numbered cell $C_{2i,2i}$ corresponds to 
$\cM_{i+1,i}$. All the remaining cells of $\cM$ are empty. To simplify the exposition even further, 
we will assume that each odd-numbered cell of the path is equal to $\Inc$ and each even-numbered 
cell is equal to~$\cD$. See Figure~\ref{fig-reduction}.

\begin{figure}
  \centering
  \includegraphics{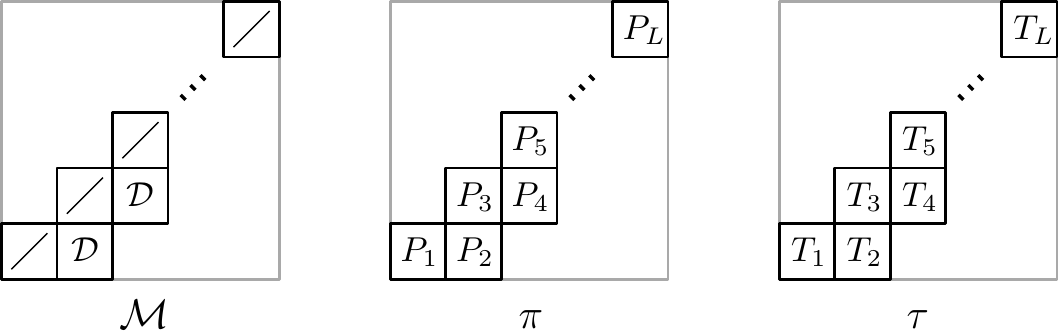}
\caption{The gridding matrix $\cM$, the gridded permutation $\pi$ (the pattern) and the gridded 
permutation $\tau$ (the text), used in the simplified overview of the proof of 
Theorem~\ref{thm-hardness}.}\label{fig-reduction}
\end{figure}

With the gridding matrix $\cM$ specified above, we will construct two $\cM$-gridded permu\-tations,
the pattern $\pi$ and the text $\tau$, such that $\pi$ can be embedded into $\tau$ if and only if 
the formula $\Phi$ is satisfiable. We will describe $\pi$ and $\tau$ geometrically, as permutation 
diagrams, which are partitioned into blocks by the $\cM$-gridding. We let $P_i$ denote 
the part of $\pi$ corresponding to the cell $C_i$ of $\cM$, and similarly we let $T_i$ be the part 
of $\tau$ corresponding to~$C_i$. 

To get an intuitive understanding of the reduction, it is convenient to first restrict our attention 
to \emph{grid-preserving embeddings} of $\pi$ into $\tau$, that is, to embeddings which map the 
elements of $P_i$ to elements of $T_i$ for each~$i$. 

The basic building blocks in the description of $\pi$ and $\tau$ are the \emph{atomic pairs}, which 
are specific pairs of points appearing inside a single block $P_i$ or~$T_i$. It is a feature of the 
construction that in any grid preserving embedding of $\pi$ into $\tau$, an atomic pair inside a 
pattern block $P_i$ is mapped to an atomic pair inside the corresponding text block~$T_i$. Moreover, 
each atomic pair in $\pi$ or $\tau$ is associated with one of the variables $x_1,\dotsc,x_n$ 
of~$\Phi$, and any grid-preserving embedding will maintain the association, that is, atomic pairs 
associated to a variable $x_j$ inside $\pi$ will map to atomic pairs associated to $x_j$ in~$\tau$.

To describe $\pi$ and $\tau$, we need to specify the relative positions of the atomic pairs in two 
adjacent blocks $P_i$ and $P_{i+1}$ (or $T_i$ and $T_{i+1}$). These relative positions are given by 
several typical configurations, which we call \emph{gadgets}. Several examples of gadgets are 
depicted in Figure~\ref{fig:gadgets}. In the figure, the pairs of points enclosed by an ellipse are 
atomic pairs. The choose, multiply and merge gadgets are used in the construction of $\tau$, 
while the pick and follow gadgets are used in~$\pi$. The copy gadget will be used in both. 
We also need more complicated gadgets, namely the \emph{flip gadgets} of 
Figure~\ref{fig:flip-gadgets}, which span more than two consecutive blocks. In all cases, the atomic 
pairs participating in a single gadget are all associated to the same variable of~$\Phi$.

\begin{figure}
  \centering
  \raisebox{-0.5\height}{\stackanchor{\includegraphics[width=0.3\textwidth]{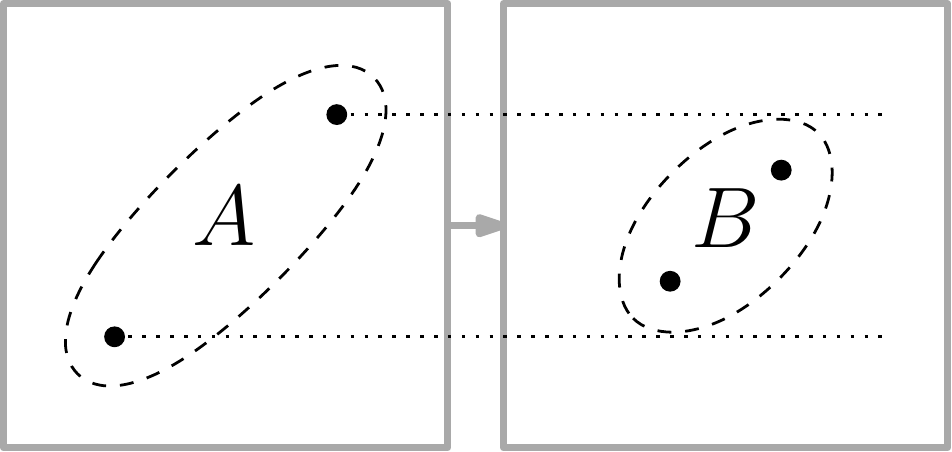}}{\small The 
copy gadget}}
  \hspace{0.15in}
  \raisebox{-0.5\height}{\stackanchor{\includegraphics[width=0.3\textwidth]{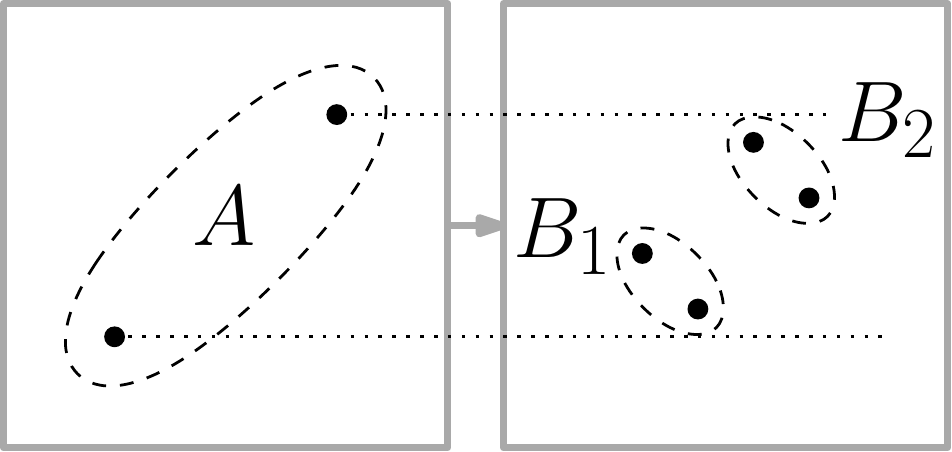}}{\small 
The choose gadget}}
  \hspace{0.15in}
  \raisebox{-0.5\height}{\stackanchor{\includegraphics[width=0.3\textwidth]{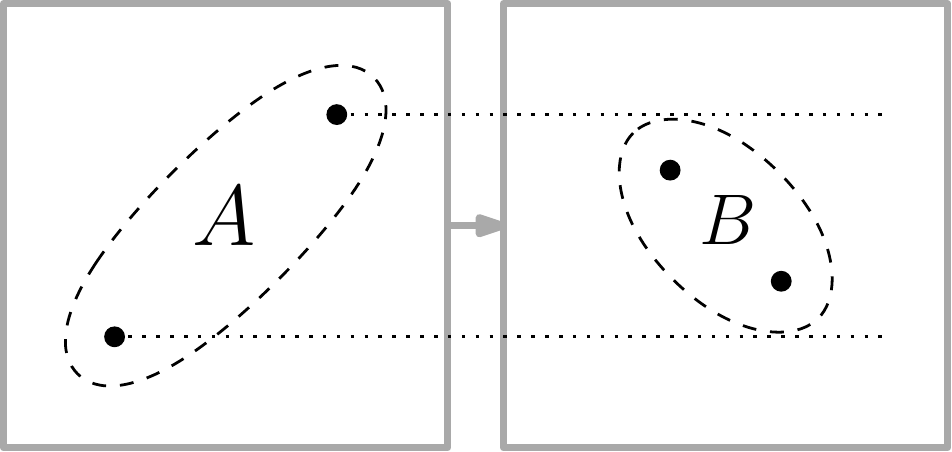}}{\small The 
pick gadget}}\\[0.1in]
  
  \raisebox{-0.5\height}{\stackanchor{\includegraphics[width=0.3\textwidth]{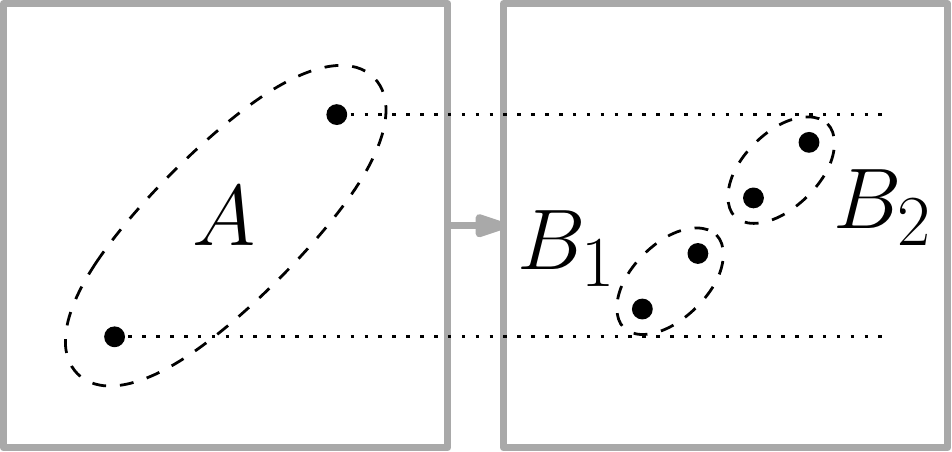}}{\small 
The multiply gadget}}
  \hspace{0.15in}
  \raisebox{-0.5\height}{\stackanchor{\includegraphics[width=0.3\textwidth]{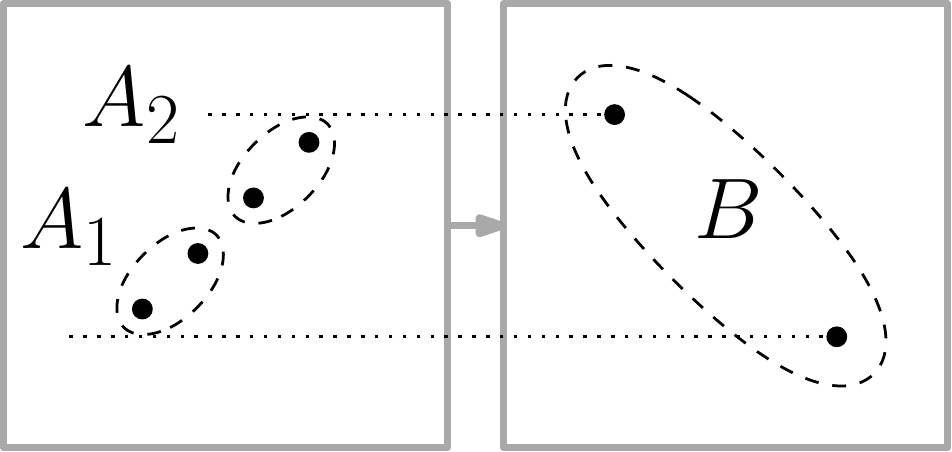}}{\small 
The merge gadget}}
  \hspace{0.15in}
  \raisebox{-0.5\height}{\stackanchor{\includegraphics[width=0.3\textwidth]{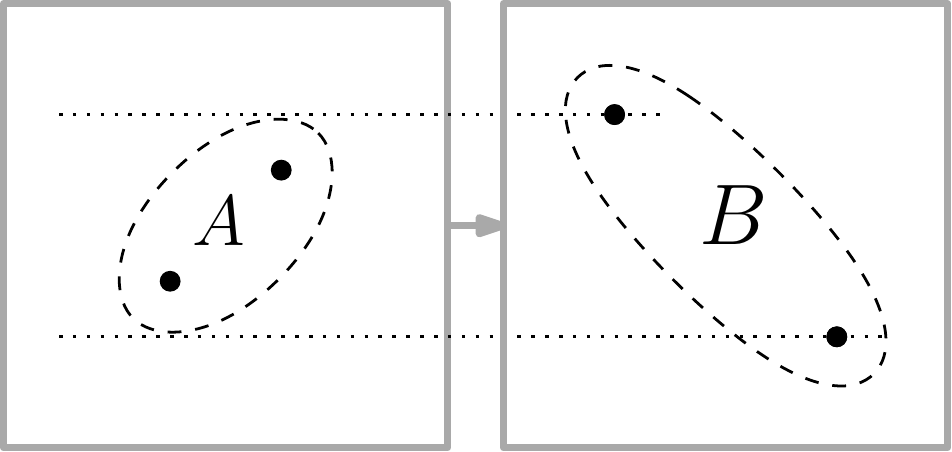}}{\small 
The follow gadget}}
  \caption{The constructions of simple gadgets. The tile $Q_i$ is always on the left and the tile 
$Q_{i+1}$ is on the right. The dotted lines show the relative vertical order of points.}
  \label{fig:gadgets}
\end{figure}

The sequence of pattern blocks $P_1, P_2,\dotsc,P_L$, as well as their corresponding text blocks 
$T_1,\dotsc,T_L$, is divided into several contiguous parts, which we call \emph{phases}. We now 
describe the individual phases in the order in which they appear.

\subparagraph{The initial phase and the assignment phase.} The initial phase involves a single 
pattern block $P_1$ and the corresponding text block $T_1$. Both $P_1$ and $T_1$ consist of an 
increasing sequence of $2n$ points, divided into $n$ consecutive atomic pairs 
$X^1_1,X^1_2, \dotsc, X^1_n\subseteq P_1$ and $Y^1_1,Y^1_2, \dotsc, Y^1_n\subseteq T_1$, numbered in 
increasing order. The pairs $X^1_j$ and $Y^1_j$ are both associated to the variable $x_j$. Clearly 
any embedding of $P_1$ into $T_1$ will map the pair $X^1_j$ to the pair $Y^1_j$, for each $j\in[n]$.

The initial phase is followed by the assignment phase, which also involves only one pattern block 
$P_2$ and the corresponding text block~$T_2$. $P_2$ will consist of an increasing sequence of $n$ 
atomic pairs $X^2_1,X^2_2, \dotsc, X^2_n$, where each $X^2_j$ is a decreasing pair, i.e., a copy of 
21. Moreover, $X^1_j\cup X^2_j$ forms the pick gadget, so the first two pattern blocks can be viewed 
as a sequence of $n$ pick gadgets stacked on top of each other. 

The block $T_2$ then consists of $2n$ atomic pairs $\{Y^2_j,Z^2_j;\;j\in[n]\}$, positioned in such a 
way that $Y^1_j\cup Y^2_j\cup Z^2_j$ is a choose gadget. Thus, $T_1\cup T_2$ is a sequence of $n$ 
choose gadgets stacked on top of each other, each associated with one of the variables of~$\Phi$. 

In a grid-preserving embedding of $\pi$ into $\tau$, each pick gadget $X^1_j\cup X^2_j$ must be 
mapped to the corresponding choose gadget $Y^1_j\cup Y^2_j\cup Z^2_j$, with $X^1_j$ mapped to 
$Y^1_j$, and $X^2_j$ mapped either to $Y^2_j$ or to $Z^2_j$. There are thus $2^n$ grid-preserving 
embeddings of $P_1\cup P_2$  into $T_1\cup T_2$, and these embeddings encode in a natural way to 
the $2^n$ assignments of truth values to the variables of~$\Phi$. Specifically, if $X^2_j$ is mapped 
to $Y^2_j$, we will say that $x_j$ is false, while if $X^2_j$ maps to $Z^2_j$, we say that $x_j$ is 
true. The aim is to ensure that an embedding of $P_1\cup P_2$ into $T_1\cup T_2$ can be extended 
to an embedding of $\pi$ into $\tau$ if and only if the assignment encoded by the embedding 
satisfies~$\Phi$.

Each atomic pair that appears in one of the text blocks $T_2, T_3,\dotsc,T_L$ is not only associated 
with a variable of $\Phi$, but also with its truth value; that is, there are `true' and `false' 
atomic pairs associated with each variable~$x_j$. The construction of $\pi$ and $\tau$ ensures that 
in an embedding of $\pi$ into $\tau$ in which $X^2_j$ is mapped to $Y^2_j$ (corresponding to setting 
$x_j$ to false), all the atomic pairs associated to $x_j$ in the subsequent stages of $\pi$ will map 
to false atomic pairs associated to $x_j$ in $\tau$, and conversely, if $X^2_j$ is mapped to 
$Z^2_j$, then the atomic pairs of $\pi$ associated to $x_j$ will only map to the true atomic pairs 
associated to $x_j$ in~$\tau$.

\begin{figure}
  \centering
  \raisebox{-0.5\height}{\includegraphics[width=0.47\textwidth]{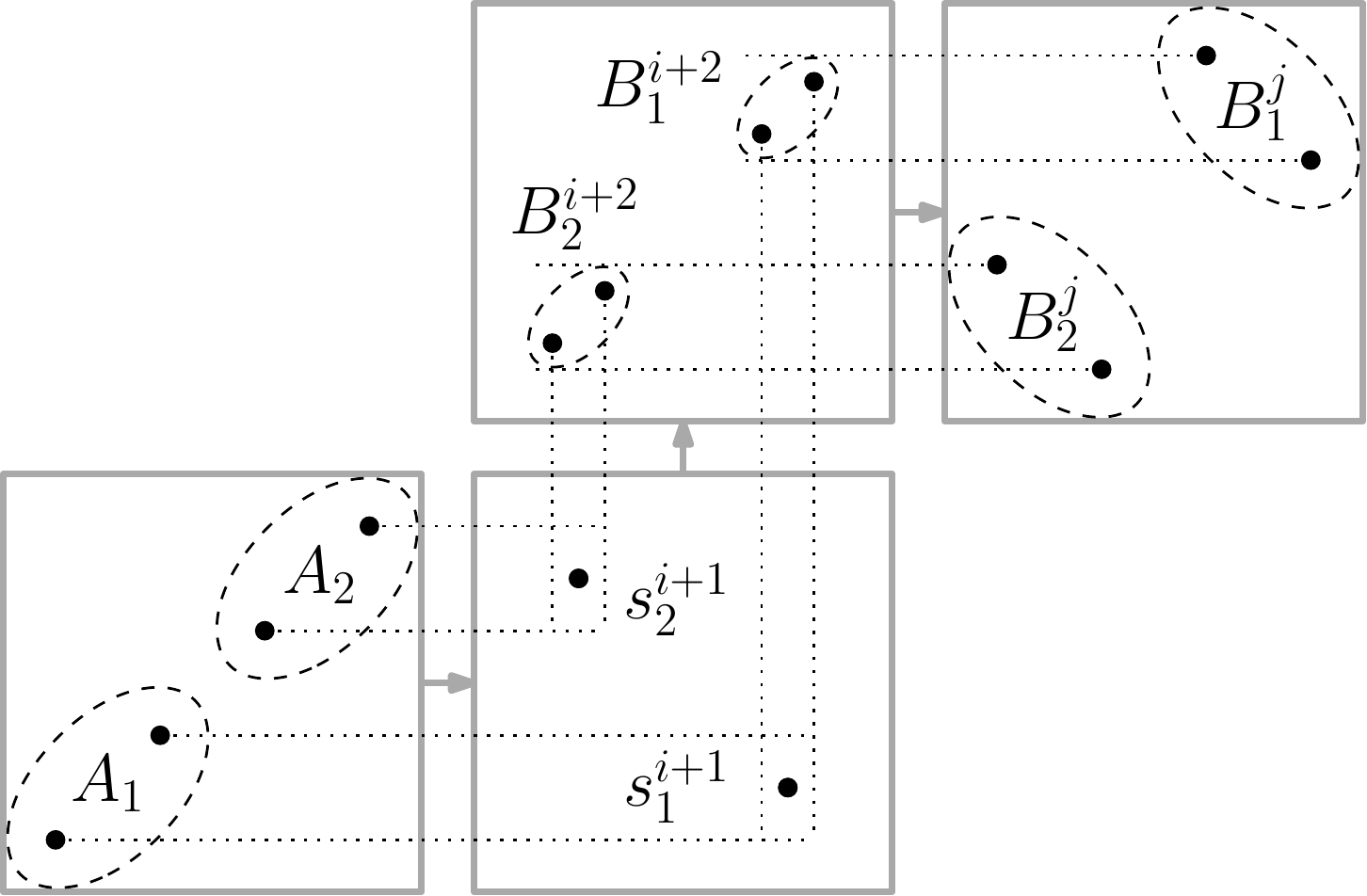}}
  \hspace{1em}\hfill
  \raisebox{-0.5\height}{\includegraphics[width=0.47\textwidth]{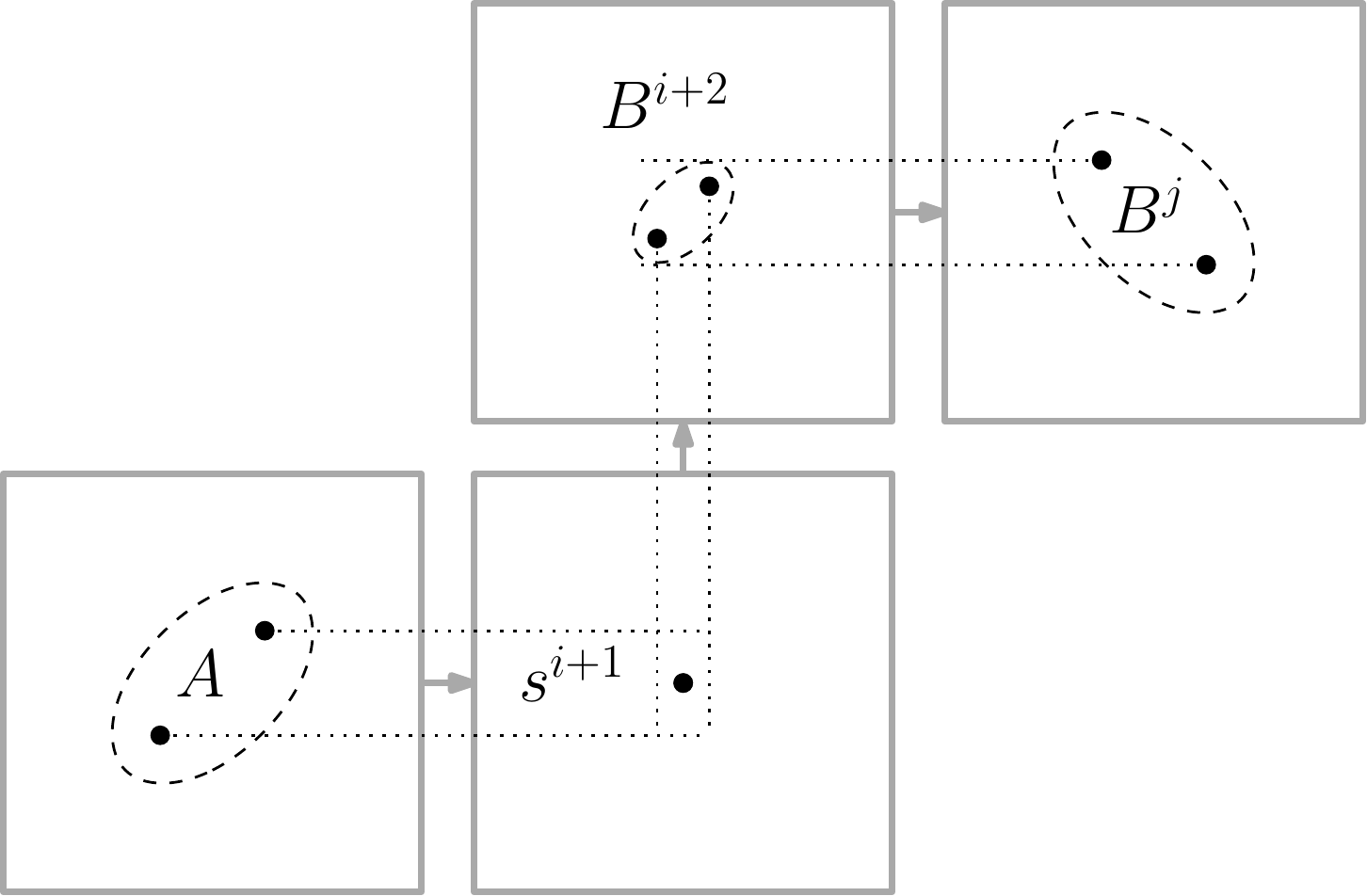}}
  \caption{A flip text gadget on the left and a flip pattern gadget on the right. The first tile 
pictured is $Q_i$ and the last tile is $Q_j$ where $j = i+3$. As before, the dotted lines show the 
relative order of points. }
  \label{fig:flip-gadgets}
\end{figure}

\subparagraph{The multiplication phase.} The purpose of the multiplication phase is to `duplicate' 
the information encoded in the assignment phase. Without delving into the technical details, we 
describe the end result of the multiplication phase and its intended behaviour with respect to 
embeddings. Let $d_j$ be the number of occurrences (positive or negative) of the variable $x_j$ 
in~$\Phi$. Note that $d_1+d_2+\dotsb+d_n=3m$, since $\Phi$ has $m$ clauses, each of them with three 
literals. Let $P_k$ and $T_k$ be the final pattern block and text block of the multiplication 
phase. Then $P_k$ is an increasing sequence of $3m$ increasing atomic pairs, among which there are 
$d_j$ atomic pairs associated to~$x_j$. Moreover, the pairs are ordered in such a way that the 
$d_1$ pairs associated to $x_1$ are at the bottom, followed by the $d_2$ pairs associated to $x_2$ 
and so on. The structure of $T_k$ is similar to $P_k$, except that $T_k$ has $6m$ atomic pairs. In 
fact, we may obtain $T_k$ from $P_k$ by replacing each atomic pair $X^k_i\subseteq P_k$ associated 
to a variable $x_j$ by two adjacent atomic pairs $Y^k_i, Z^k_i$, associated to the same variable, 
where $Y^k_i$ is false and $Z^k_i$ is true. 

It is useful to identify each pair $X^k_i\subseteq P_k$ as well as the corresponding two pairs 
$Y^k_i, Z^k_i\subseteq T_k$ with a specific occurrence of $x_j$ in~$\Phi$. Thus, each literal in 
$\Phi$ is represented by one atomic pair in $P_k$ and two adjacent atomic pairs of opposite truth 
values in~$T_k$.

The blocks $P_3,\dotsc,P_k$ and $T_3,\dotsc,T_k$ are constructed in such a way that any embedding of 
$\pi$ into $\tau$ that encodes an assignment in which $x_j$ is false has the property that all the 
atomic pairs in $P_k$ associated to $x_j$ are mapped to the false atomic pairs of $T_k$ associated to 
$x_j$, and similarly, when $x_j$ is encoded as true in the assignment phase, the pairs of $P_k$ 
associated to $x_j$ are only mapped to the true atomic pairs of $T_k$. Thus, the mapping of any 
atomic pair of $P_k$ encodes the information on the truth assignment of the associated variable.

The multiplication phase is implemented by a combination of multiply gadgets and flip text gadgets 
in $\tau$, and copy gadgets and flip pattern gadgets in~$\pi$. It requires no more than $O(\log m)$ 
blocks in $\pi$ and $\tau$, i.e., $k=O(\log m)$.

\subparagraph{The sorting phase.} The purpose of the sorting phase is to rearrange the relative 
positions of the atomic pairs. While at the end of the multiplication phase, the pairs representing 
occurrences of the same variable appear consecutively, after the sorting phase, the pairs 
representing literals belonging to the same clause will appear consecutively. More precisely, 
letting $P_\ell$ and $T_\ell$ denote the last pattern block and the last text block of the sorting 
phase, $P_\ell$ has the same number of atomic pairs associated to a given variable $x_j$ as $P_k$, 
and similarly for $T_\ell$ and~$T_k$. If $K_1,\dotsc,K_m$ are the clauses of $\Phi$, then for each 
clause $K_j$, $P_\ell$ contains three consecutive atomic pairs corresponding to the three literals in 
$K_j$, and $T_\ell$ contains the corresponding six atomic pairs, again appearing consecutively. 
Similarly as in $P_k$ and $T_k$, each atomic pair in $P_\ell$ must map to an atomic pair in $T_\ell$ 
representing the same literal and having the correct truth value encoded in the assignment phase.

To prove Theorem~\ref{thm-hardness}, we need two different ways to implement the sorting phase, 
depending on whether the class $\cD$ contains a monotone juxtaposition or not. The first 
construction, which we call \emph{sorting by gadgets}, does not put any extra assumptions on~$\cD$. 
However, it may require up to $\Theta(m)$ blocks to perform the sorting, that is $\ell=\Theta(m)$. 

The other implementation of the sorting phase, which we call \emph{sorting by juxtapositions} is 
only applicable when $\cD$ contains a monotone juxtaposition, and it can be performed with only 
$O(\log m)$ blocks. The difference between the lengths of the two versions of sorting is the reason 
for the two different lower bounds in Theorem~\ref{thm-hardness}.

\subparagraph{The evaluation phase.} The final phase of the construction is the evaluation phase. 
The purpose of this phase is to ensure that for any embedding of $\pi$ into $\tau$, the truth 
assignment encoded by the embedding satisfies all the clauses of~$\Phi$. For each clause $K_j$, we 
attach suitable gadgets to the atomic pairs in $P_\ell$ and $T_\ell$ representing the literals 
of~$K_j$. Using the fact that the atomic pairs representing the literals of a given clause are 
consecutive in $P_j$ and~$T_j$, this can be done for all the clauses simultaneously, with only 
$O(1)$ blocks in $\pi$ and~$\tau$.  This completes an overview of the hardness reduction proving 
Theorem~\ref{thm-hardness}.

When the reduction is performed with sorting by gadgets, it produces permutations $\pi$ and $\tau$ 
of size $O(m^2)$, since we have $L=O(m)$ blocks and each block has size~$O(m)$. When sorting is done 
by juxtapositions, the number of blocks drops to $L=O(\log m)$, hence $\pi$ and $\tau$ have size 
$O(m\log m)$. ETH implies that 3-SAT with $n$ variables and $m$ clauses cannot be solved in 
time $2^{o(m+n)}$~\cite{IPZ}. From this, the lower bounds from Theorem~\ref{thm-hardness} follow.

The details of the reduction, as well as the full correctness proof, are presented in the following subsections.

\subsection{Details of the hardness reduction}\label{ssec-reduction}

Our job is to construct a pair of permutations $\pi$ and $\tau$, both having a gridding 
corresponding to a $\cD$-rich path, with the property that the embeddings of $\pi$ into $\tau$ will 
simulate satisfying assignments of a given 3-SAT formula. We will describe the two permutations in 
terms of their diagrams, or more precisely, in terms of point sets order-isomorphic to the diagrams, 
constructed inside the prescribed gridding. 

The formal description of the point sets faces several purely technical complications: firstly, 
some monotone cells of the given $\cD$-rich path correspond to $\Inc$, others to $\Dec$, and it 
would be a nuisance to distinguish the two symmetric possibilities in every step of the construction. 
We will sidestep the issue by fixing a consistent orientation of the gridding which will, informally 
speaking, allow us to treat the monotone cells as if they were all increasing.

Another complication stems from the fact that when describing the contents of a given 
cell in the gridding, the coordinates of the points in the cell depend on the position of the cell 
within the gridding matrix. It would be more convenient to assume that each cell has its own local 
coordinate system whose origin is in the bottom-left corner of the cell. We address this
issue by treating each cell as an independent `tile', and only in the end of the construction we 
translate the individual tiles (after applying appropriate symmetries implied by the consistent 
orientation) to their proper place in the gridding.

To make these informal considerations rigorous, we now introduce the concept of $\cF$-assembly.

\subsubsection{\texorpdfstring{$\cF$}{F}-assembly}

Recall that a \emph{$k \times \ell$ orientation} is a pair of functions $\cF = (f_c, f_r)$, $f_c\colon [k] \to \{-1,1\}$ and $f_r\colon [\ell] \to \{-1,1\}$.

A finite subset $P$ of the $m$-box in general position is called an \emph{$m$-tile} and a \emph{$k 
\times \ell$ family of $m$-tiles} is a set $\cP = \{P_{i,j} \mid i \in [k], j\in [\ell]\}$ where 
each $P_{i,j}$ is an $m$-tile. Let $\cF = (f_c, f_r)$ be a $k \times \ell$ orientation and let $\cP$ be 
a family of $m$-tiles $P_{i,j}$ for $i \in [k]$, $j \in [\ell]$. The \emph{$\cF$-assembly of $\cP$} 
is the point set $S$ defined as follows.

First, we define for every $i 
\in [k]$, $j \in [\ell]$ the point set
\[P'_{i,j} = \{p + (i \cdot m, j \cdot m) \mid p \in \Phi_i(\Psi_j(P_{i,j}))\}.\]
where $\Phi_i$ is an identity if $f_c(i) = 1$ and reversal otherwise, while $\Psi_j$ is an identity 
if $f_r(j) = 1$ and complement otherwise.
Next, we set $S = \bigcup P'_{i,j}$. The set $S$ might not be 
in a general position. If 
that is the case, we rotate $S$ clockwise by a tiny angle such that there are no longer any points 
with a shared coordinate, and at the same time, the order of all the other points is preserved. 
We call the resulting point set  the \emph{$\cF$-assembly} of~$\cP$.
See Figure~\ref{fig:F-assembly}.

% Let $\cM$ be a $k \times \ell$ gridding matrix. We say that the \emph{image of $\cM_{i,j}$ under $\cF$}
% is the class $\Phi_i(\Psi_j(\cM_{i,j}))$.
% The \emph{$\cF$-image of $\cM$}, denoted by $\cF(\cM)$, is then the $k 
% \times \ell$ gridding matrix defined as $\cF(\cM)_{i,j} = \Phi_i(\Psi_j(\cM_{i,j}))$.
% Observe that this operation is an involution as $\cF(\cF(\cM)) = \cM$.

% \begin{observation}
% Let $\cM$ be a $k \times l$ gridding matrix, let $\cF$ be a $k \times \ell$ orientation and $\cP$ a 
% $k \times \ell$ family of $m$-tiles. If for every $i \in [k]$ and $j \in [\ell]$ 
% the reduction of $P_{i,j}$ belongs to the class $\cF(\cM)_{i,j}$ then the reduction of the 
% $\cF$-assembly of  $\cP$ belongs to $\Grid(\cM)$.
% \end{observation}
The $\cF$-assembly allows us to describe our constructions more succinctly and independently from 
the actual matrix $\cM$. Recall that an orientation $\cF$ of a monotone gridding matrix $\cM$ is 
\emph{consistent}, if every non-empty entry of $\cF(\cM)$ is equal to~$\Inc$. Gridding matrices 
whose cell graph is a path always admit a consistent orientation -- this follows from the following 
more general result of Vatter and Waton~\cite{Vatter2011}.

\begin{lemma}[Vatter and Waton~\cite{Vatter2011}]
\label{lem-acyclictrans}
Any monotone gridding matrix $\cM$ whose cell graph $G_\cM$ is acyclic has a consistent orientation.
\end{lemma}

\begin{figure}
  \centering
  $\begin{array}{cc}
    P_{1,2} = \vcenter{\hbox{\includegraphics[scale=0.45]{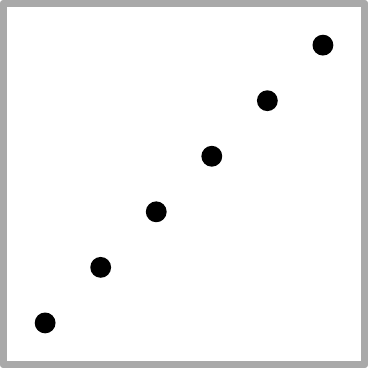}}}  & 
    P_{2,2} = \vcenter{\hbox{\includegraphics[scale=0.45]{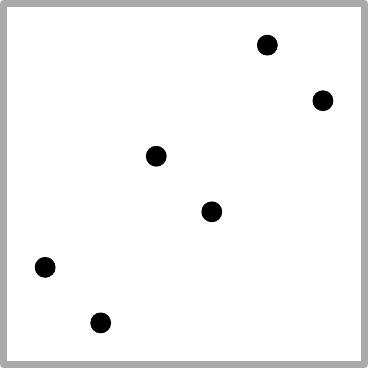}}}\\[0.3in]
    P_{1,1} = \vcenter{\hbox{\includegraphics[scale=0.45]{inc-tile}}} &
    P_{2,1} = \vcenter{\hbox{\includegraphics[scale=0.45]{fib-tile}}} 
  \end{array}$
  \hspace{0.2in}
  $\longrightarrow$
  \hspace{0.2in}
  \raisebox{-0.5\height}{\includegraphics[scale=0.45]{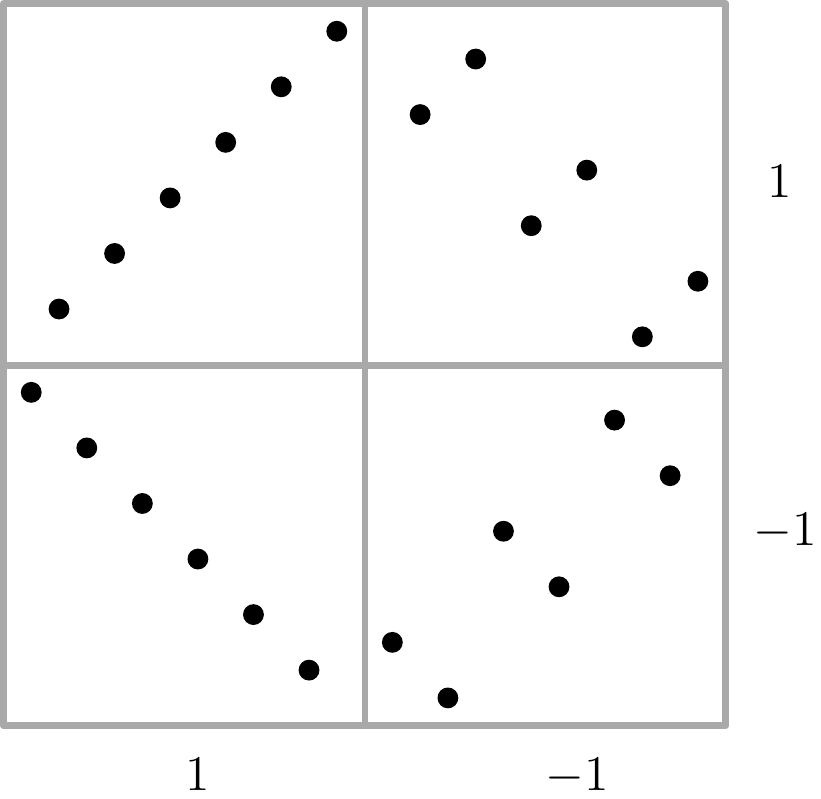}}
  \caption[Example of an $\cF$-assembly]{
A $2 \times 2$ family of tiles $\cP$ on the left and its $\cF$-assembly on the right for 
a $2 \times 2$ orientation $\cF$ given next to each row and column on the right. Notice that the
$\cF$-assembly of $\cP$ would belong to the class 
$\Grid \begin{psmallmatrix} \Inc & \ominus 12\\ \Dec & \oplus21 \end{psmallmatrix}$ 
for any $\cP$ such that $P_{1,1}, P_{1,2} \in \Inc$ and $P_{2,1},P_{2,2} \in \oplus21$.
}\label{fig:F-assembly}
\end{figure}

\subsubsection{The reduction}

We describe a reduction from 3-SAT to \CPPM{$\cC$}.  Let $\Phi$ be a given 
3-CNF formula with $n$ variables $x_1, x_2,\dotsc,x_n$ and 
$m$ clauses $K_1, K_2,\dotsc,K_m$. We furthermore assume, without loss of generality, 
that each clause contains exactly three literals. We will construct permutations $\pi, \tau \in \cC$ 
such that $\Phi$ is satisfiable if and only if $\tau$ contains $\pi$.

Let $\cM$ be a $g \times h$ gridding matrix such that $\Grid(\cM)$ is a subclass of $\cC$, the cell 
graph $G_\cM$ is a proper-turning path of sufficient length to be determined later, a constant 
fraction of its entries is equal to $\cD$, and the remaining entries of the path are monotone. We 
aim to construct $\pi$ and $\tau$ such that they both belong to $\Grid(\cM)$. We remark that this is 
the only step where the computable part of the computable $\cD$-rich path property comes into play.

First, we label the vertices of the path as $p_1, p_2, p_3, \dotsc$ choosing the direction such that 
at least half of the $\cD$-entries share a row with their predecessor. We claim that there is a $g 
\times h$ orientation $\cF$ such that the class $\cF(\cM)_{i,j}$ is equal to $\Inc$ for every monotone entry 
$\cM_{i,j}$  and the class $\cF(\cM)_{i,j}$ contains $\oplus 21$ for every $\cD$-entry $\cM_{i,j}$.
To see this, consider a gridding matrix $\cM'$ obtained by replacing every 
$\cD$-entry in $\cM$ with $\Inc$ if $\cD$ contains $\oplus 21$ or with $\Dec$ if $\cD$ contains 
$\ominus 12$. It then suffices to apply Lemma~\ref{lem-acyclictrans} on $\cM'$.

Our plan is to simultaneously construct two $g \times h$ families of $2n$-tiles $\cP$ and $\cT$ and 
then set $\pi$ and $\tau$ to be the $\cF$-assemblies of $\cP$ and $\cT$, respectively. We abuse the 
notation and for any $g \times h$ family of tiles $\cQ$ (in particular for $\cP$ and $\cT$), we use 
$Q_i$ instead of $Q_{p_i}$ to denote the tile corresponding to the $i$-th cell of the path.

For now, we will only consider restricted embeddings that respect the partition into tiles and deal 
with arbitrary embeddings later. Let $\pi$ be an 
$\cF$-assembly of $\cP$ and $\tau$ an $\cF$-assembly of $\cT$; we then say that an embedding of 
$\pi$ into $\tau$ is \emph{grid-preserving} if the 
image of tile $P_{i,j}$ is mapped to the image of $T_{i,j}$ for every $i$ and $j$. We slightly abuse 
the terminology in the case of grid-preserving embeddings and say that a point $q$ in the tile 
$P_{i,j}$ is mapped to a point $r$ in the tile $T_{i,j}$ instead of saying that the image of $q$ 
under the $\cF$-assembly is mapped to the image of $r$ under the $\cF$-assembly.

Many of the definitions to follow are stated for a general $g \times h$ family of tiles $\cQ$ as we 
later apply them on both $\cP$ and $\cT$. We say that a pair of points $r, q$ in the tile $Q_i$ 
\emph{sandwiches} a set of points $A$ in the tile $Q_{i+1}$ if for every point $t \in A$
\begin{itemize}%[noitemsep]
  \item $r.y < t.y < q.y$ in case $p_i$ and $p_{i+1}$ occupy a common row, or
  \item $r.x < t.x < q.x$ in case $p_i$ and $p_{i+1}$ occupy a common column.
\end{itemize}

Moreover, we say that a pair of points $r, q$ \emph{strictly sandwiches} a set of points $A$ if the 
pair $r, q$ sandwiches $A$ and there exists no other point $t \in( Q_{i+1} \setminus A)$ sandwiched 
by $r,q$.

\subsubsection{Simple gadgets}

We construct the tiles in $\cT$ and $\cP$ from pairs of points that we call \emph{atomic pairs}. 
Before constructing the actual tile families, we describe several simple gadgets that we will 
utilize later. All these gadgets include either one atomic pair $A = (r,q)$ or two atomic pairs 
$A_\alpha = (r_\alpha, q_\alpha)$ for $\alpha \in \{1,2\}$ in the tile $Q_i$ and one atomic pair $B 
= (s, t)$ or two atomic pairs $B_\alpha = (s_\alpha, t_\alpha)$ for $\alpha \in \{1,2\}$ in the tile 
$Q_{i+1}$. Moreover, the coordinates of the points in $Q_{i+1}$ are fully determined from the 
coordinates of the points in $Q_i$. We assume that each tile is formed as direct sum of the 
corresponding pieces of individual gadgets; in other words, if $A,B \subseteq Q_i$ are point sets 
of two different gadgets, then either $A$ lies entirely to the right and above $B$ or vice versa.

We describe the gadgets in the case when $p_i$ and $p_{i+1}$ share a common row and $p_i$ is left 
from $p_{i+1}$, as the other cases are symmetric. Moreover, we assume that $r.y < q.y$ in the case 
of a single atomic pair in $Q_i$, and that $r_1.y < q_1.y < r_2.y < q_2.y$ in the case of two 
atomic pairs in~$Q_i$. See Figure~\ref{fig:gadgets}.

\subparagraph{The copy gadget.}
The copy gadget consists of one atomic pair $A$ in $Q_i$ and one atomic pair $B$ in $Q_{i+1}$ 
defined as
\[s = (r.y, r.y + \eps) \qquad t = (q.y, q.y - \eps)\]
where $\eps$ is small positive value such that $s,t$ form an occurrence of 12. We say that the copy 
gadget connects $A$ to $B$. 
Since $A$ sandwiches $B$, we can use the copy gadget to extend the construction to additional tiles 
along the path while preserving the embedding properties.

\begin{observation}
  \label{obs-copy-gadget}
  Suppose there is a copy gadget in $\cT$ that connects an atomic pair $A^T$ in the tile $T_{i}$ to 
an atomic pair $B^T$ in the tile $T_{i+1}$, and a copy gadget in $\cP$ that connects an atomic pair 
$A^P$ in the tile $P_i$ to an atomic pair $B^P$ in the tile $P_{i+1}$. In any grid-preserving 
embedding of $\pi$ into $\tau$, if $A^P$ is mapped to $A^T$, then $B^P$ is mapped to $B^T$.
\end{observation}

\subparagraph{The multiply gadget.}
The multiply gadget is only slightly more involved than the previous one, and it consists of one 
atomic pair $A$ in $Q_i$ and two atomic pairs $B_1, B_2$ in the tile $Q_{i+1}$ defined as
\begin{align*}
  \begin{aligned}[c]
    s_1 &= (r.y, r.y + \eps)\\
    t_1 &= \left(\frac{2 \cdot r.y + q.y}{3}, \frac{2 \cdot r.y + q.y}{3}\right)
  \end{aligned}
  \qquad
  \begin{aligned}[c]
    s_2 &=\left(\frac{r.y + 2 \cdot q.y}{3}, \frac{r.y + 2 \cdot q.y}{3}\right)\\
    t_2 &= (q.y, q.y - \eps)
  \end{aligned}
\end{align*}
where $\eps$ is small positive value such that $s_1, t_1, s_2, t_2$ form an occurrence of $1234$. We 
say that the multiply gadget multiplies $A$ to $B_1$ and $B_2$.  A 
property analogous to Observation~\ref{obs-copy-gadget} holds when both text and pattern contain a 
multiply gadget.

% \begin{figure}
%   \centering
%   \raisebox{-0.5\height}{\stackanchor{\includegraphics[width=0.3\textwidth]{copy-gadget}}{\small The 
% copy gadget}}
%   \hspace{0.15in}
%   \raisebox{-0.5\height}{\stackanchor{\includegraphics[width=0.3\textwidth]{choose-gadget}}{\small 
% The choose gadget}}
%   \hspace{0.15in}
%   \raisebox{-0.5\height}{\stackanchor{\includegraphics[width=0.3\textwidth]{pick-gadget}}{\small The 
% pick gadget}}\\[0.1in]
%   
%   \raisebox{-0.5\height}{\stackanchor{\includegraphics[width=0.3\textwidth]{multiply-gadget}}{\small 
% The multiply gadget}}
%   \hspace{0.15in}
%   \raisebox{-0.5\height}{\stackanchor{\includegraphics[width=0.3\textwidth]{merge-gadget}}{\small 
% The merge gadget}}
%   \hspace{0.15in}
%   \raisebox{-0.5\height}{\stackanchor{\includegraphics[width=0.3\textwidth]{follow-gadget}}{\small 
% The follow gadget}}
%   \caption{The constructions of simple gadgets. The tile $Q_i$ is always on the left and the tile 
% $Q_{i+1}$ is on the right. The dotted lines show the relative vertical order of points.}
%   \label{fig:gadgets}
% \end{figure}

\subparagraph{The choose gadget.}
The choose gadget  consists of one atomic pair $A$ in $Q_i$ and two atomic pairs $B_1, B_2$ in the 
tile $Q_{i+1}$ defined as
\begin{align*}
  \begin{aligned}[c]
    s_1 &= \left(r.y, \frac{2 \cdot r.y + q.y}{3}\right)\\
    t_1 &= \left(\frac{2 \cdot r.y + q.y}{3}, r.y + \eps \right)\\
  \end{aligned}
  \qquad
  \begin{aligned}[c]
    s_2 &=\left(\frac{r.y + 2 \cdot q.y}{3}, q.y - \eps \right)\\
    t_2 &= \left(q.y, \frac{r.y + 2 \cdot q.y}{3} \right).
  \end{aligned}
\end{align*}
where $\eps$ is small positive value such that $s_1, t_1, s_2, t_2$ form an occurrence of $2143$. We 
say that the choose gadget branches $A$ to $B_1$ and $B_2$. 

\subparagraph{The pick gadget.}
The pick gadget is essentially identical to the copy gadget except that the pair $B$ forms an 
occurrence of $21$ instead of $12$. Formally, the pick gadget consists of one atomic pair $A$ in 
$Q_i$ and one atomic pair $B$ in $Q_{i+1}$ defined as
\begin{align*}
  s = \left(r.y, q.y - \eps\right) \qquad t = \left(q.y, r.y + \eps \right)
\end{align*}
where $\eps$ is chosen such that $s,t$ form an occurrence of 21. We say that the pick gadget 
connects $A$ to $B$. 

The name of the choose and pick gadgets becomes clear with the following observation that follows 
from the fact that 2143 admits only two possible embeddings of the pattern 21.
\begin{observation}
  \label{obs-choose-gadget}
  Suppose there is a choose gadget in $\cT$ that branches an atomic pair $A^T$ in the tile $T_{i}$ 
to two atomic pairs $B^T_1$ and $B^T_2$ in the tile $T_{i+1}$, and a pick gadget in $\cP$ that 
connects an atomic pair $A^P$ in the tile $P_i$ to an atomic pair $B^P$ in the tile $P_{i+1}$. In 
any grid-preserving embedding of $\pi$ into $\tau$, if $A^P$ is mapped to $A^T$ then $B^P$ is mapped 
either to $B^T_1$ or to $B^T_2$.
\end{observation}

\subparagraph{The merge gadget.}
The merge gadget consists of two atomic pairs $A_1$, $A_2$ in $Q_i$ and one atomic pair $B$ in 
$Q_{i+1}$ defined as
\begin{align*}
  s = \left(r_1.y, q_2.y + \eps\right) \qquad t = \left(q_2.y, r_1.y - \eps \right)
\end{align*}
where $\eps$ is chosen small such that the relative order of the points of merge gadget and the 
points outside is not changed. Notice that now $B$ sandwiches the set $A_1 \cup A_2$.  We say that 
the merge gadget merges $A_1$ and $A_2$ into $B$.

\subparagraph{The follow gadget.}
This gadget is almost identical to the pick gadget except that here $B$ sandwiches $A$. Formally, 
the follow gadget contains one atomic pair $A$ in $Q_i$ and one atomic pair $B$ in $Q_{i+1}$ defined 
as
\begin{align*}
  s = \left(r.y, q.y + \eps\right) \qquad t = \left(q.y, r.y - \eps \right)
\end{align*}
where $\eps$ is chosen small such that the relative order of follow gadget with points outside is 
not changed. We say that the follow gadget connects $A$ to $B$. 

We can observe that merge and follow gadgets act in a way as an inverse to choose and pick gadgets. 
\begin{observation}
  \label{obs-merge-gadget}
  Suppose there is a merge gadget in $\cT$ that merges atomic pairs $A^T_1$ and $A^T_2$ in the tile 
$T_i$ into an atomic pair $B^T$ in the tile $T_{i+1}$, and a follow gadget in $\cP$ that connects an 
atomic pair $A^P$ in the tile $P_i$ to an atomic pair $B^P$ in the tile $P_{i+1}$. In any 
grid-preserving embedding of $\pi$ into $\tau$, if $A^P$ is mapped to $A^T_\alpha$ for some $\alpha 
\in \{1,2\}$ then $B^P$ is mapped to $B^T$. 
\end{observation}
To see this, notice that $B^P$ forms an occurrence of 21 and the only occurrence of 21 in $T_{i+1}$ 
that sandwiches $A^T_1$ or $A^T_2$ is $B^T$. Here it is important that $T_{i+1}$ is formed as a 
direct sum of the individual gadgets and in particular, every other occurrence of 21 in $T_{i+1}$ 
lies either above or below $B^T$.

\subsubsection{The flip gadget}

We proceed to define two gadgets --- a flip text gadget and a flip pattern gadget.
The construction of this final pair of gadgets is a bit more involved. It is insufficient to 
consider just two neighboring tiles as we need two $\cD$-entries for the construction. To that end, 
let $i$ and $j$ be indices such that both $p_{i+1}$ and $p_j$ are $\cD$-entries and every entry 
$p_k$ for $k$ between $i+1$ and $j$ is a monotone entry. Recall that every $\cD$-entry shares a row 
with its predecessor. In particular, $p_i$ occupies the same row as $p_{i+1}$ and $p_{j-1}$ occupies 
the same row as $p_{j}$.  

As before, suppose that $A_1 = (r_1, q_1)$ and $A_2 = (r_2, q_2)$ are two atomic pairs in $Q_i$ such 
that $r_1.y < q_1.y < r_2.y < q_2.y$. The \emph{flip text gadget} attached to the atomic pairs $A_1$ 
and $A_2$ consists of two points $s^{i+1}_1, s^{i+1}_2$ in the tile $Q_{i+1}$ and two atomic pairs 
$B^k_1 = (s^{k}_1, t^{k}_1), B^k_2=(s^k_2, t^k_2)$ in each tile $Q_k$ for every $k \in [i+2,j] = 
\{i+2, i+3, \dots, j\}$. The points $s^{i+1}_1, s^{i+1}_2$ are defined as
\begin{align*}
  s^{i+1}_1 = \left(\frac{r_1.y + q_1.y}{2}, \frac{r_2.y + q_2.y}{2} \right) \quad s^{i+1}_2 = 
\left(\frac{r_2.y + q_2.y}{2}, \frac{r_1.y + q_1.y}{2} \right).
\end{align*}
Observe that $s^{i+1}_1, s^{i+1}_2$ form an occurrence of 21 such that $s^{i+1}_\alpha$ is 
sandwiched by the pair $A_\alpha$ for each $\alpha \in \{1,2\}$. The points $s^{k}_1, t^{k}_1, 
s^k_2, t^k_2$ for $k \in [i+2, j-1]$ are defined as
\begin{align*}
  \begin{aligned}[c]
    s^k_1 &= \left(r_2.y + \eps, r_2.y\right)\\
    t^k_1 &= \left(q_2.y - \eps, q_2.y \right)
  \end{aligned}
  \qquad
  \begin{aligned}[c]
    s^k_2 &= \left(r_1.y + \eps, r_1.y\right)\\
    t^k_2 &= \left(q_1.y - \eps, q_1.y\right).
  \end{aligned}
\end{align*}
if $p_k$ and $p_{k+1}$ share the same column, otherwise we just apply the adjustments by $\eps$ in 
the $y$-coordinates. The positive constant $\eps$ is chosen such that the points $s^k_2, t^k_2, 
s^{k}_1, t^{k}_1$ (in this precise left-to-right order) form an occurrence of 1234 for every $k$ 
between $i+1$ and $j$.

Finally, the points $s^{j}_1, t^{j}_1, s^j_2, t^j_2$ are defined as
\begin{align*}
  \begin{aligned}[c]
    s^j_1 &= \left(r_2.y, q_2.y\right)\\
    t^j_1 &= \left(q_2.y, r_2.y \right)
  \end{aligned}
  \qquad
  \begin{aligned}[c]
    s^j_2 &=\left(r_1.y, q_1.y\right)\\
    t^j_2 &= \left(q_1.y, r_1.y \right).
  \end{aligned}
\end{align*}
Observe that they form an occurrence of 2143. We say that the flip text gadget flips the pairs $A_1, 
A_2$ in $Q_i$ to the pairs $B^j_2, B^j_1$ in $Q_j$. See the left part of 
Figure~\ref{fig:flip-gadgets}.

We defined the points such that for every $k$ between $i+1$ and $j$ and $\alpha \in \{1,2\}$, the 
pair $B^{k+1}_\alpha$ sandwiches the pair $B^k_{\alpha}$. Alternatively, this can be seen as 
$B^{k+1}_\alpha$ and $B^k_\alpha$ forming a copy gadget alas in the opposite direction. Moreover, 
the pair $B^{i+2}_\alpha$ sandwiches the point $s^{i+1}_\alpha$.

Observe that independently of the actual orientation $\cF$, the images of all points $s^k_1, t^k_1$ 
for all $k$ and the images of all points $s^k_2, t^k_2$ for all $k$ under the $\cF$-assembly will be 
isomorphic. We define the flip pattern gadget as a set of points isomorphic to this particular set 
of points.

Given an atomic pair $A = (r,q)$ in $Q_i$, the \emph{flip pattern gadget} attached to the atomic 
pair $A$ consists of a single point $s^{i+1}$ in the tile $Q_{i+1}$ and an atomic pair $B^k = (s^k, 
t^k)$ in the tile $Q_k$ for every $k \in [i+2, j]$, where

\begin{align*}
  s^{i+1} = \left(\frac{r.y + q.y}{2}, \frac{r.y + q.y}{2} \right) ,
\end{align*}
the atomic pair $B^k$ is defined for every $k \in [i+2, j-1]$  as
\begin{align*}
  s^k =\left(r.y + \eps, q.y\right) \qquad t^k = \left(q.y - \eps, r.y \right) 
\end{align*}
if $p_k$ and $p_{k+1}$ share a common column, otherwise we just apply the adjustments by $\eps$ in 
the $y$-coordinate, and finally, the atomic pair $B^j$ is defined as
\begin{align*}
  s^j =\left(r.y, q.y\right) \qquad t^j = \left(q.y, r.y \right).
\end{align*}
We say that the flip pattern gadget connects the pair $A$ to the pair $B^j$. See the right part of 
Figure~\ref{fig:flip-gadgets}.

% \begin{figure}
%   \centering
%   \raisebox{-0.5\height}{\includegraphics[width=0.48\textwidth]{flip-text-gadget}}
%   \hspace{1em}
%   \raisebox{-0.5\height}{\includegraphics[width=0.48\textwidth]{flip-pattern-gadget}}
%   \caption{A flip text gadget on the left and a flip pattern gadget on the right. The first tile 
% pictured is $Q_i$ and the last tile is $Q_j$ where $j = i+3$. As before, the dotted lines show the 
% relative order of points. }
%   \label{fig:flip-gadgets}
% \end{figure}

\begin{lemma}
  \label{lem:flip-gadget}
  Suppose there is a flip pattern gadget in $\cP$ that connects an atomic pair $\overline{A}$ in 
$P_i$ with an atomic pair $\overline{B}$ in $P_{j}$. Furthermore, suppose that there is a flip text 
gadget in $\cT$ that flips atomic pairs $A_1$ and $A_2$ in $T_i$ to atomic pairs $B_2$ and $B_1$ in 
$T_{j}$. In any grid-preserving embedding of $\pi$ into $\tau$, if $\overline{A}$ is mapped to 
$A_\alpha$ for some $\alpha \in \{1,2\}$ then $\overline{B}$ is mapped to $B_\alpha$.
\end{lemma}
\begin{proof}
  We denote the points of both gadgets as in their respective definitions. Additionally, we use 
overlined letters to denote points of the flip pattern gadget in $\cP$ to distinguish them from the 
points of the flip text gadget in $\cT$.
  
  Observe that $\overline{s}^{i+1}$ must be mapped to $s^{i+1}_\alpha$. This implies that the point 
$\overline{s}^{i+2}$ must be mapped to the point $s^{i+2}_\alpha$ or below and the point 
$\overline{t}^{i+2}$ must be mapped to the point $t^{i+2}_\alpha$ or above. By repeating this 
argument, we see that $\overline{s}^{k}$ must be mapped to the point $s^{k}_\alpha$ or below and the 
point $\overline{t}^{k}$ must be mapped to the point $t^{k}_\alpha$ or above for every $k \in [i+2, 
j]$. But the only occurrence of 21 in $T_j$ with this property is precisely the pair $B_\alpha$ 
which concludes the proof. We remark that here we again use the property that $T_j$ can be expressed 
as a direct sum of the individual gadgets. Otherwise, we could find a suitable occurrence of 21 in 
$T_j$ as part of a different gadget.
\end{proof}

Flip gadgets will see two slightly different applications. First, as the name suggests, a flip 
gadget allows us to shuffle the order of atomic pairs.
Second and perhaps more cunning use of the flip gadget is that it allows us to test if only one of 
its initial atomic pairs is used in the embedding.

\begin{lemma}
  \label{lem-flip-test}
  Suppose that there are two flip pattern gadgets in $\cP$ each connecting an atomic pair 
$A^P_\alpha$ in $P_i$ to an atomic pair $B^P_\alpha$ in $P_j$  for $\alpha \in \{1,2\}$. Suppose 
that there is a flip text gadget in $\cT$ that flips atomic pairs $A^T_1$ and $A^T_2$ in $T_i$ to 
atomic pairs $B^T_2$ and $B^T_1$ in $T_j$. There cannot exist a grid-preserving embedding $\phi$ of 
$\pi$ into $\tau$ that maps $A^P_\alpha$ to $A^T_\alpha$ for each $\alpha \in \{1,2\}$.
\end{lemma}
\begin{proof}
  Using Lemma~\ref{lem:flip-gadget}, we see that $\phi$ would map also $B^P_\alpha$ to $B^T_\alpha$ 
for each $\alpha \in \{1,2\}$. But that is a contradiction since $B^T_1$ lies to right and above 
$B^T_2$ in $T_j$ while $B^P_1$ lies to the left and below $B^P_2$ in $P_j$.
\end{proof}

We conclude the introduction of gadgets with one more definition. All the gadgets except for the 
copy and multiply ones need the entry $p_{i+1}$ to be non-monotone, more precisely the image of 
$\cM_{p_{i+1}}$ under $\cF$ has to contain the Fibonacci class $\oplus 21$. Suppose there is an atomic pair $A$ 
in the tile $Q_i$ and that $j$ is the smallest index larger than $i$ such that $p_j$ is a 
$\cD$-entry. By \emph{attaching a gadget} other than copy or multiply to the pair $A$, we mean the 
following procedure. We add an atomic pair $A^k$ to each tile $Q_k$ for $k \in [i+1,j-1]$ such that 
$A^{k}$ and $A^{k+1}$ form a copy gadget for each $k \in [i,j-2]$ when we additionally define $A^i = 
A$. Finally, we attach the desired gadget to the atomic pair $A^{j-1}$. Similarly, when attaching a 
gadget that contains two atomic pairs $A_1$ and $A_2$ in its first tile, we just copy both of these 
pairs all the way to the tile $Q_{j-1}$ and then attach the desired gadget. It follows from 
Observation~\ref{obs-copy-gadget} that the embedding properties are preserved via this construction.

\subsubsection{Constructing the \texorpdfstring{\CPPM{$\cC$}}{C-PPM} instance}

We define the initial tile $P_1$ to contain atomic pairs $X^0_k = (q_k, r_k)$ for $k\in [n]$ and the 
initial tile $T_1$ to contain atomic pairs $Y^0_k = (s_k, t_k)$ for $k\in [n]$ where
\[q_k = s_k = (2k-1, 2k -1) \qquad r_k = t_k = (2k, 2k).\]
Any grid-preserving embedding of $\pi$ into $\tau$ must obviously map $X^0_k$ to $Y^0_k$ for every 
$k \in [n]$.

We describe the rest of the construction in four distinct phases. We start by simulating the 
assignment of truth values with two possible mappings for each of the $n$ atomic pairs in the 
assignment phase. In the multiplication phase, we manufacture an atomic pair for each occurrence of 
a variable in a clause while keeping the possible mappings of all pairs corresponding to the same 
variable consistent. In the following sorting phase, we rearrange the atomic pairs to be bundled 
together by their clauses. And finally, in the evaluation phase, we test in parallel that each 
clause is satisfied.

\subparagraph{Assignment phase.}In the first phase, we simulate the assignment of truth values to 
the variables. To that end, we attach to each pair $Y^0_k$ for $k \in [n]$ a choose gadget that 
branches $Y^0_k$ to two atomic pairs $Y^1_{k,1}$ and $Z^1_{k,1}$. On the pattern side, we attach to 
each pair $X^0_k$ for $k \in [n]$ a pick gadget that connects $X^0_k$ to an atomic pair $X^1_{k,1}$. 
The properties of choose and pick gadgets imply that in any grid-preserving embedding, $X^1_{k,1}$ 
is either mapped to $Y^1_{k,1}$ or to $Z^1_{k,1}$.

\subparagraph{Multiplication phase.}Our next goal is to multiply the atomic pairs corresponding to a 
single variable into as many pairs as there are occurrences of this variable in the clauses. We 
describe the gadgets dealing with each variable individually.

Fix $k \in [n]$ and let $m_k$ for $k \in [n]$ denote the total number of occurrences of $x_k$ and 
$\neg x_k$ in $\Phi$. We are going to describe the construction inductively in $\ell_k = \lceil \log 
m_k \rceil$ steps.
In $i$-th step, we define atomic pairs $X^{i+1}_{k,j}, Y^{i+1}_{k,j},Z^{i+1}_{k,j}$ for $j \in 
[2^{i+1}]$ such that in any grid-preserving embedding, $X^{i+1}_{k,j}$ maps either to 
$Y^{i+1}_{k,j}$ or to  $Z^{i+1}_{k,j}$. Moreover, the order of the atomic pairs in the pattern tile 
is $X^{i+1}_{k,1}, X^{i+1}_{k,2}, \dotsc, X^{i+1}_{k,2^{i+1}}$ and the order of the atomic pairs in 
the text tile is
$$Y^{i+1}_{k,1}, Z^{i+1}_{k,1},Y^{i+1}_{k,2}, Z^{i+1}_{k,2},  \dotsc, Y^{i+1}_{k,2^{i+1}}, 
Z^{i+1}_{k,2^{i+1}}.$$

First, notice that the properties hold for $i = 0$ at the end of the assignment phase.
Fix $i \ge1$. We add for each $j \in [2^i]$ three multiply gadgets, one that multiplies the atomic 
pair $X^i_{k,j}$ to atomic pairs $\widetilde{X}^{i+1}_{k,2j-1}$ and $\widetilde{X}^{i+1}_{k,2j}$, 
one that multiplies the pair $Y^i_{k,j}$ to $\widetilde{Y}^{i+1}_{k,2j-1}$ and 
$\widetilde{Y}^{i+1}_{k,2j}$, and finally one that multiplies $Z^i_{k,j}$ to 
$\widetilde{Z}^{i+1}_{k,2j-1}$ and $\widetilde{Z}^{i+1}_{k,2j}$. Observe that the properties of 
gadgets together with induction imply that for arbitrary $j \in [2^{i+1}]$,  
$\widetilde{X}^{i+1}_{k,j}$ maps either to  $\widetilde{Y}^{i+1}_{k,j}$ or to 
$\widetilde{Z}^{i+1}_{k,j}$. Moreover, the atomic pairs $\widetilde{X}^{i+1}_{k,j}$ are already 
ordered in the pattern by $j$ as desired. However, the order of the atomic pairs in text is 
incorrect as for each $j \in [2^{i}]$ we have the quadruple
\[\widetilde{Y}^{i+1}_{k,2j-1}, \widetilde{Y}^{i+1}_{k,2j}, \widetilde{Z}^{i+1}_{k,2j-1}, 
\widetilde{Z}^{i+1}_{k,2j} \]
in this specific order.

To solve this, we add for each $j \in [2^i]$ a flip text gadget that flips 
$\widetilde{Y}^{i+1}_{k,2j}$, $\widetilde{Z}^{i+1}_{k,2j-1}$ to atomic pairs $Z^{i+1}_{k,2j-1}, 
Y^{i+1}_{k,2j}$. Furthermore, we attach a flip pattern gadget to every other atomic pair in both 
pattern and text. In particular, we add one that connects the pair $\widetilde{Y}^{i+1}_{k,2j-1}$ to 
a pair $Y^{i+1}_{k,2j-1}$, one that connects $\widetilde{Z}^{i+1}_{k,2j}$ to a pair $Z^{i+1}_{k,2j}$ 
and finally two that connect $\widetilde{X}^{i+1}_{k,\alpha}$ to $X^{i+1}_{k, \alpha}$ for $\alpha 
\in \{2j-1, 2j\}$. The properties of flip gadgets guarantee that for every $j \in [2^{i+1}]$, the 
pair $X^{i+1}_{k, j}$ is mapped either to  $Y^{i+1}_{k, j}$ or to $Z^{i+1}_{k, j}$. Moreover, the 
order of atomic pairs in the text now alternates between $Y$ and $Z$ as desired.

We described the gadget constructions independently for each variable. The unfortunate effect is 
that we might have used up a different total number of tiles for each variable. We describe a way to 
fix this. Let $k$ be such that $m_k$ is the largest value and let $j$ be the largest index such that 
the tiles $T_j$ and $P_j$ have been used in the multiplication phase for the $k$-th variable. For 
every other variable, we simply attach a chain of copy gadgets connecting every pair at the end of 
its multiplication phase all the way to the tiles $P_j$ and $T_j$. Observe that we need in total 
$O(\log m)$ entries equal to $\cD$ for the multiplication phase as $m_k$ is at most $3m$.

\subparagraph{Sorting phase via gadgets.}
The multiplication phase ended with atomic pairs $X^{i}_{k,j}$ in the pattern and $Y^{i}_{k,j}$ and 
$Z^{i}_{k,j}$ in the text for some $i$, every $k \in [n]$ and $j \in [m_k]$.\footnote{In fact for 
every $j \in [2^{\lceil \log m_k \rceil}]$ but we simply ignore the pairs for $j > m_k$.} These 
pairs are ordered lexicographically by $(k,j)$, i.e., they are bundled in blocks by the variables. 
The goal of the sorting phase, as the name suggests, is to rearrange them such that they become 
bundled by clauses while retaining the embedding properties.

We show how to use gadgets to swap two neighboring atomic pairs in the pattern. Suppose that $X_1$ 
and $X_2$ are two atomic pairs in some tile $P_i$ such that $X_1$ is to the left and below $X_2$ and 
all the remaining atomic pairs are either to the right and top of both $X_1, X_2$ or to the left and 
below both $X_1, X_2$. Suppose that $Y_1, Y_2, Z_1, Z_2$ are atomic pairs in the tile $T_i$ such 
that they are ordered in $T_i$ as  $Y_1, Z_1, Y_2, Z_2$, and every other atomic pair lies either to 
the right and top or to the left and below of all of them. Furthermore, suppose that in any 
grid-preserving embedding $\phi$, the pair $X_\alpha$ is mapped either to $Y_\alpha$ or to 
$Z_\alpha$ for each $\alpha \in \{1,2\}$. Notice that this is precisely the case at the end of the 
multiplication phase.

We attach choose gadgets to each of the pairs $Y_1, Z_1, Y_2, Z_2$ and a pick gadget to both $X_1, 
X_2$. In particular, we add for every $\alpha \in \{1,2\}$
\begin{itemize}%[noitemsep]
  \item a choose gadget that branches $Y_\alpha$ to pairs $\overline{Y}_\alpha$ and 
$\widetilde{Y}_\alpha$,
  \item a choose gadget that branches $Z_\alpha$ to pairs $\overline{Z}_\alpha$ and 
$\widetilde{Z}_\alpha$, and
  \item a pick gadgets that connects $X_\alpha$ to $\overline{X}_\alpha$.
\end{itemize}

We will now abuse our notation slightly and use the same letters to denote atomic pairs in different 
tiles so that the names are carried through with the flip gadgets. In other words, a flip text 
gadget shall flip atomic pairs $A_1, A_2$ to pairs $A_2, A_1$ and a flip pattern gadget connects 
atomic pair $A$ to an atomic pair $A$.
Using three layers of flip gadgets in $\cT$, we change the order of the pairs in the following way. 
Note that the blue color is used to highlight pairs to which $\overline{X}_1$ can be mapped, and red 
color for the pairs to which $\overline{X}_2$ can be mapped and arrows show which pairs are flipped 
in each step.
\begin{align*}
  \colb{\overline{Y}_1} \colb{\widetilde{Y}_1} \colb{\overline{Z}_1} \flip{\colb{\widetilde{Z}_1} 
\colr{\overline{Y}_2}} \flip{\colr{\widetilde{Y}_2} \colr{\overline{Z}_2}} \colr{\widetilde{Z}_2} 
&\rightarrow
  \colb{\overline{Y}_1} \colb{\widetilde{Y}_1} \flip{\colb{\overline{Z}_1}  \colr{\overline{Y}_2}} 
\flip{\colb{\widetilde{Z}_1} \colr{\overline{Z}_2}} \colr{\widetilde{Y}_2} \colr{\widetilde{Z}_2} 
\rightarrow
  \colb{\overline{Y}_1} \flip{\colb{\widetilde{Y}_1}  \colr{\overline{Y}_2}} 
\flip{\colb{\overline{Z}_1}  \colr{\overline{Z}_2}}  \flip{\colb{\widetilde{Z}_1} 
\colr{\widetilde{Y}_2}} \colr{\widetilde{Z}_2}\\
  &\rightarrow \colb{\overline{Y}_1}  \colr{\overline{Y}_2}  \colb{\widetilde{Y}_1}   
\colr{\overline{Z}_2} \colb{\overline{Z}_1}  \colr{\widetilde{Y}_2} \colb{\widetilde{Z}_1} 
\colr{\widetilde{Z}_2}
\end{align*}
On the pattern side, we attach three consecutive copies of flip pattern gadget to both 
$\overline{X}_1$ and $\overline{X}_2$.

Now we perform the actual swap by attaching a flip text gadget that flips $\overline{X}_1$ and 
$\overline{X}_2$ and attaching flip text gadget between each of the four neighboring pairs in the 
text, i.e.
\begin{align*}
  \flip{\colb{\overline{Y}_1}  \colr{\overline{Y}_2}}  \flip{\colb{\widetilde{Y}_1}   
\colr{\overline{Z}_2}} \flip{\colb{\overline{Z}_1}  \colr{\widetilde{Y}_2}} 
\flip{\colb{\widetilde{Z}_1} \colr{\widetilde{Z}_2}} \rightarrow
  \colr{\overline{Y}_2} \colb{\overline{Y}_1} \colr{\overline{Z}_2}  \colb{\widetilde{Y}_1} 
\colr{\widetilde{Y}_2} \colb{\overline{Z}_1}  \colr{\widetilde{Z}_2} \colb{\widetilde{Z}_1}.
\end{align*}

As a next step, we unshuffle the pairs to which $\overline{X}_1$ and $\overline{X}_2$ can be mapped, 
again by three layers of flip gadgets
\begin{align*}
  \colr{\overline{Y}_2} \flip{\colb{\overline{Y}_1} \colr{\overline{Z}_2}}  
\flip{\colb{\widetilde{Y}_1} \colr{\widetilde{Y}_2}} \flip{\colb{\overline{Z}_1}  
\colr{\widetilde{Z}_2}} \colb{\widetilde{Z}_1} &\rightarrow
  \colr{\overline{Y}_2}  \colr{\overline{Z}_2} \flip{\colb{\overline{Y}_1}   \colr{\widetilde{Y}_2}} 
\flip{\colb{\widetilde{Y}_1}   \colr{\widetilde{Z}_2}} \colb{\overline{Z}_1} \colb{\widetilde{Z}_1} 
\rightarrow
  \colr{\overline{Y}_2}  \flip{\colr{\overline{Z}_2}  \colr{\widetilde{Y}_2}} 
\flip{\colb{\overline{Y}_1} \colr{\widetilde{Z}_2}} \colb{\widetilde{Y}_1}  \colb{\overline{Z}_1} 
\colb{\widetilde{Z}_1}\\
  &\rightarrow 
  \colr{\overline{Y}_2} \colr{\widetilde{Y}_2} \colr{\overline{Z}_2}     \colr{\widetilde{Z}_2} 
\colb{\overline{Y}_1} \colb{\widetilde{Y}_1}  \colb{\overline{Z}_1} \colb{\widetilde{Z}_1}
\end{align*}
As before, we attach three consecutive copies of flip pattern gadget to both $\overline{X}_1$ and 
$\overline{X}_2$.

Finally, we add for every $\alpha \in \{1,2\}$
\begin{itemize}%[noitemsep]
  \item a merge gadget that merges $\overline{Y}_\alpha$ and $\widetilde{Y}_\alpha$ to a pair  
$Y'_\alpha$,
  \item a merge gadget that merges $\overline{Z}_\alpha$ and $\widetilde{Z}_\alpha$ to a pair  
$Z'_\alpha$, and
  \item a follow gadget that connects $\overline{X}_\alpha$ to $X'_\alpha$.
\end{itemize}

It follows from the properties of the individual gadgets that in any grid-preserving embedding, the 
pair $X'_\alpha$ is mapped either to $Y'_\alpha$ or $Z'_\alpha$ for each $\alpha \in \{1,2\}$. On 
the other hand, any grid-preserving embedding of the first $i$ tiles can be extended to the points 
added in the construction. The crucial observation is that in the step when the order of 
$\overline{X}_1$ and $\overline{X}_2$ is reversed, four neighboring pairs of atomic pairs are 
flipped in $\cT$ which form all possible combinations of $Y_1, Z_1$ and $Y_2, Z_2$. Therefore, we 
can always choose where to map the pairs $\overline{X}_1$ and $\overline{X}_2$ at the first step to 
arrive at one of these pairs that get flipped. 

Notice that we can use the described construction to swap arbitrary subset of disjoint neighboring 
atomic pairs in the pattern in parallel. It is easy to see that we can sort any sequence of length 
$\ell$ using at most $\ell$ rounds of such parallel swaps. 
Since the total number of atomic pairs in the pattern after the multiplication phase is exactly 
$3m$, the number of rounds needed is at most $3m$. Individually, each of the $3m$ steps uses only a 
constant number of layers of gadgets and thus only a constant amount of $\cD$-entries. Therefore, 
the sorting phase takes in total $O(m)$ $\cD$-entries.

\subparagraph{Sorting phase via juxtapositions.}
We claim that the sorting phase can be done using significantly fewer $\cD$-entries if the class 
$\cD$ contains a monotone juxtaposition. 

Suppose that at the beginning of the sorting phase there are atomic pairs $X_1, X_2, \ldots, X_{3m}$ 
in the pattern and atomic pairs $Y_1, Z_1, Y_2, Z_2, \ldots, Y_{3m}, Z_{3m}$ in the text (both in 
this precise order) such that in any grid-preserving mapping, $X_{i}$ is mapped either to $Y_{i}$ or 
to $Z_{i}$.
Our goal is to rearrange the pairs so that there are atomic pairs $\overline{X}_{\sigma_1}, 
\overline{X}_{\sigma_2}, \ldots, \overline{X}_{\sigma_{3m}}$ in the pattern and atomic pairs 
$\overline{Y}_{\sigma_1},\overline{Z}_{\sigma_1}, \overline{Y}_{\sigma_2}, \overline{Z}_{\sigma_2}, 
\ldots, \overline{Y}_{\sigma_{3m}}, \overline{Z}_{\sigma_{3m}}$ in the text in this order given by 
permutation $\sigma$ of length $3m$.
Moreover, in any grid-preserving embedding $\overline{X}_i$ is mapped to $\overline{Y}_i$ if $X_i$ 
is mapped to $Y_i$, and it is mapped to $\overline{Z}_i$ if $X_i$ is mapped to $Z_i$.

Suppose that the proper-turning path $p_1, p_2, \ldots$ contains $L$ entries equal to $\cD$. Since 
there are in total only 4 possible images of $\cD$ given by the orientation $\cF$, there exists a 
monotone juxtaposition $\cB$ and at least $L/4$ indices $j$ such that $\cM_{p_j} = \cD$ and the 
class $\cF(\cM)_{p_j}$ contains $\cB$. We are going to use only these entries for the 
sorting phase.

First, suppose that $\cB = \Grid(\Inc \, \Inc)$. Let $p_i$ be an entry such that $\cF(\cM)_{p_i}$ contains $\cB$ 
and recall that we consider only those non-monotone entries that share a common row with their 
predecessor, i.e., $p_i$ shares a common row with $p_{i-1}$. We can construct a tile $Q_i$  from two 
tiles $Q^1_i$ and $Q^2_i$ where both $Q^1_i$ and $Q^2_i$ contain an increasing point set and $Q^1_i$ 
and $Q^2_i$ are placed next to each other. In particular, we can then attach to any atomic pair $A$ 
in $Q_{i-1}$ a copy gadget connecting $A$ to an atomic pair $B$ and choose arbitrarily whether $B$ 
lies in $Q^1_i$ or $Q^2_i$.

Let $J_1$ and $J_2$ be a partition of the set $[3m]$. We attach a copy gadget ending in $Q^\alpha_i$ 
to each $X_j, Y_j$ and $Z_j$ with $j \in J_\alpha$ for each $\alpha \in \{1,2\}$. In this way, we 
rearranged the atomic pairs in $\cP$ such that first we have all pairs $X_j$ such that $j \in J_1$ 
(sorted by the indices) followed by all pairs $X_j$ for $j \in J_2$ (again themselves sorted by the 
indices). Similarly in $\cT$, we first have $Y_j, Z_j$ for $j \in J_1$ followed by $Y_j, Z_j$ for 
$j\in J_2$. See Figure~\ref{fig:sort}.

Notice that the described operation simulates a stable bucket sort with two buckets. Therefore, we 
can simulate radix sort and rearrange the atomic pairs into arbitrary order given by $\sigma$ by 
iterating this operation $O(\log m)$ times. In this way, the whole sorting phase uses only $O(\log 
m)$ entries equal to $\cD$.

\begin{figure}
  \centering
  \raisebox{-0.5\height}{\includegraphics[scale=0.4]{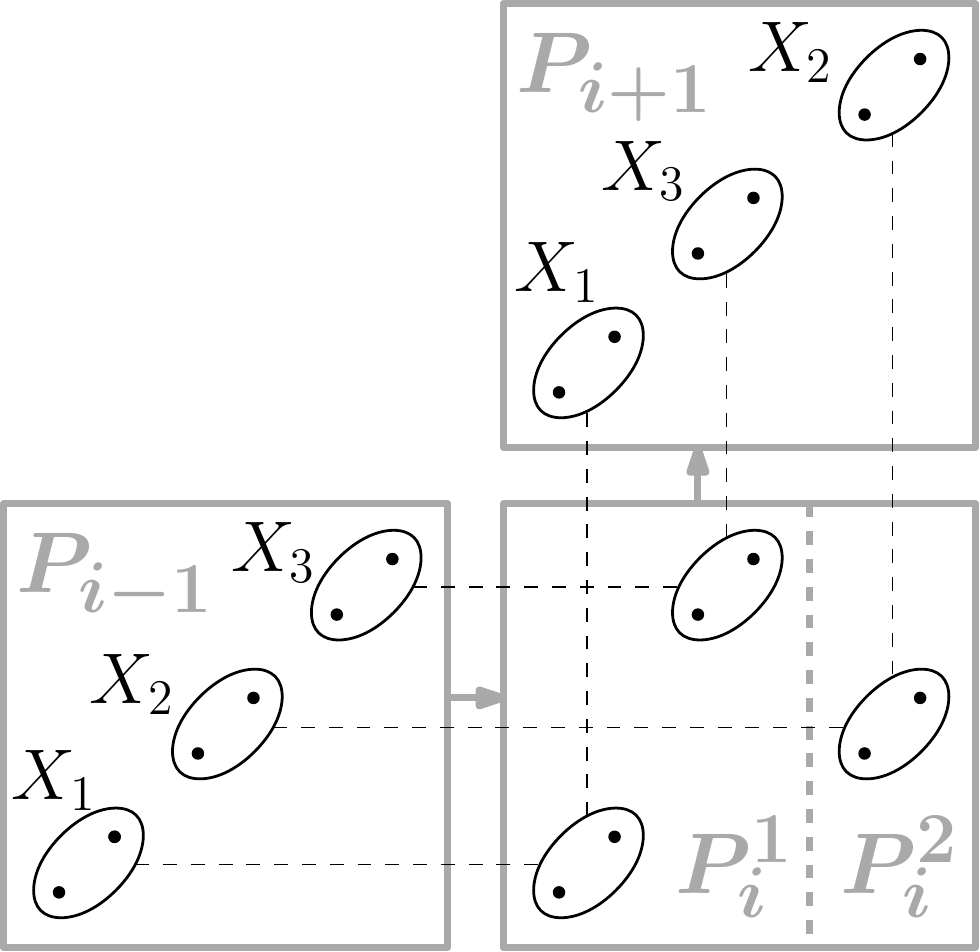}}
  \hspace{1in}
  \raisebox{-0.5\height}{\includegraphics[scale=0.4]{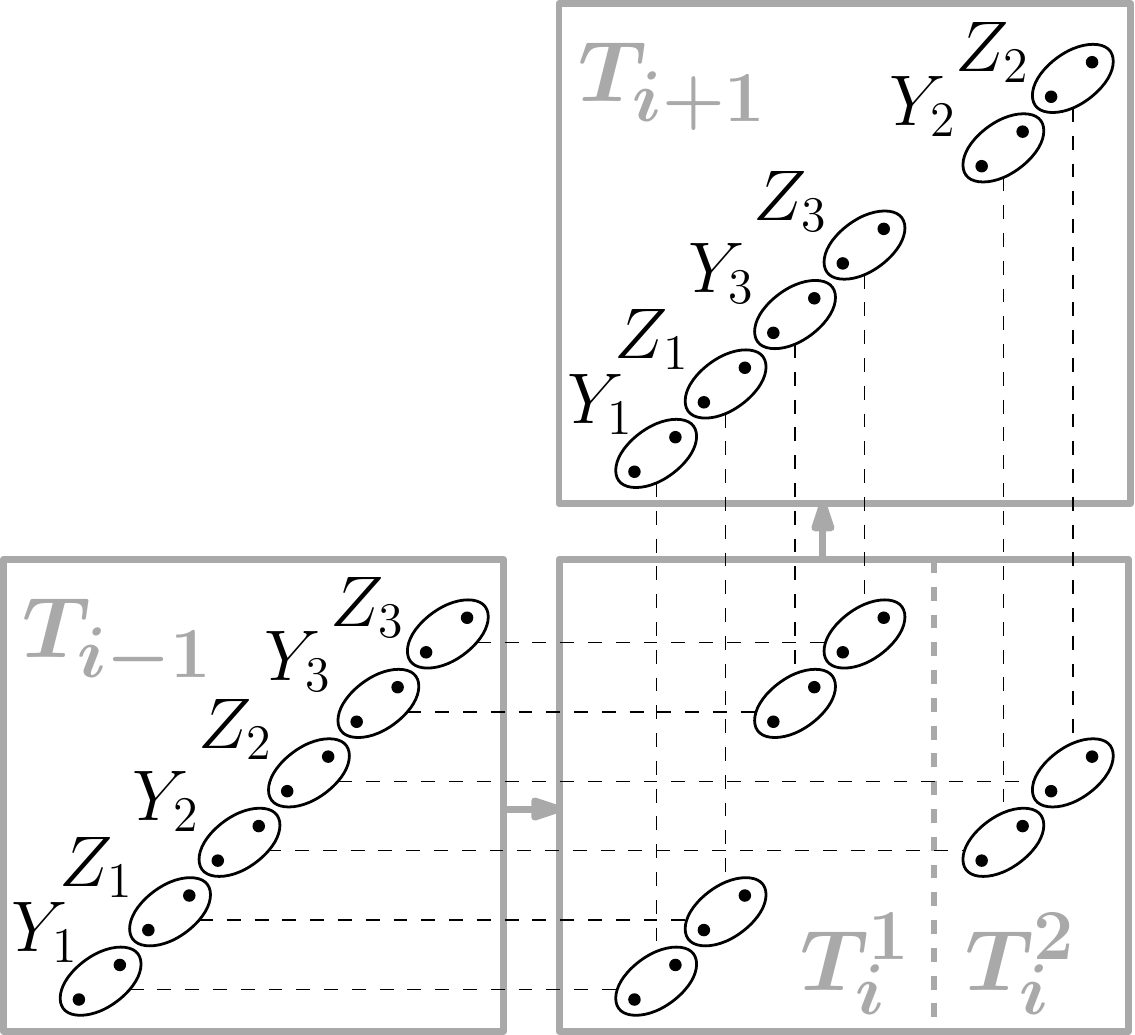}}
  \caption[Example of sorting by juxtaposition]{Example of one sorting step using the juxtaposition 
$\cB = \Grid(\Inc\,\Inc)$ and 
partition of the set $\{1,2,3\}$ into $J_1 = \{1,3\}$ and $J_2 = \{2\}$. The pattern tiles are on 
the left, the text tiles are on the right.}
  \label{fig:sort}
\end{figure}

Now suppose that $\cB = \Grid(\cA_1 \; \cA_2)$ where $\cA_1$, $\cA_2$ are two arbitrary monotone 
classes. We can use the same construction as before. However, we need to be careful that some of 
$Q^1_i$ and $Q^2_i$ should actually contain a decreasing sequence instead of increasing. Following 
the procedure as before, we still have in $\cP$ first all the pairs $X_j$ for $j \in J_1$ followed 
by the pairs $X_j$ for $j \in J_2$. However, the order of pairs $X_j$ for $j \in J_\alpha$ is now 
reversed if $\cA_\alpha = \Dec$. This can be fixed by using one extra entry whose image under $\cF$ is $\cB$. 
We partition $[3m]$ into $J'_1 = J_1$ and $J'_2 = J_2$ and attach the same construction once again. 
The property that every $X_j$ for $j \in J_1$ precedes every $X_j$ for $j \in J_2$ is preserved and 
moreover, any of the two parts that were reversed in the first step is now ordered again in the 
correct increasing order by the indices. Thus, we again implemented a stable bucket sort with two 
buckets and we can rearrange the atomic pairs into arbitrary order using $O(\log m)$ such steps.

Now, suppose that $\cB = \Grid\begin{psmallmatrix}\cA_2 \\ \cA_1\end{psmallmatrix}$ for two monotone 
classes $\cA_1$, $\cA_2$. In this case, we can again construct the tile $Q_i$ from two tiles $Q^1_i$ 
and $Q^2_i$ where both $Q^1_i$ and $Q^2_i$ contain a monotone point set determined by the classes 
$\cA_1$, $\cA_2$ but $Q^1_i$ and $Q^2_i$ are this time placed on top of each other. First, suppose 
that $\cA_1 = \cA_2 = \Inc$.

This time we can choose how to split the set $[3m]$ into two sets of consecutive numbers $J_1$ and 
$J_2$ such that every element of $J_1$ is smaller than every element of $J_2$ and again connect a 
copy gadget ending in $Q^\alpha_i$ to each $X_i$, $Y_j$, and $Z_j$ with $j \in J_\alpha$ for each 
$\alpha \in \{1,2\}$. However, this time, we also choose the relative order of gadgets between 
$Q^1_i$ and $Q^2_i$, which enables us to arbitrarily interleave the sequences of atomic pairs 
indexed by $J_1$ and $J_2$, respectively. Effectively, we implemented an inverse operation to the 
stable bucket sort with two buckets -- we split the sequence of atomic pairs into two (uneven) 
halves and then interleave them arbitrarily while keeping each of the two parts in the 
original order. Therefore, we can again rearrange the atomic pairs into an arbitrary order iterating 
this operation $O(\log m)$ times and thus using only $O(\log m)$ entries equal to $\cD$.

Finally, it remains to deal with the case when $\cB = \Grid\begin{psmallmatrix}\cA_2 \\ 
\cA_1\end{psmallmatrix}$ and $\cA_1$, $\cA_2$ are arbitrary monotone classes. Observe that the 
construction described in the previous paragraph results in interleaving the two sequences and 
simultaneously reversing the order of pairs in $J_\alpha$ if $\cA_\alpha = \Dec$. We can easily 
resolve this by prepending an extra step of the same construction. We first split $[3m]$ into the 
same halves $J'_1 = J_1$ and $J'_2 = J_2$ and use the described construction. However, we do not 
interleave the two sets and keep them in the same order. Therefore, we have only reversed the order 
of pairs in the half $J_\alpha$ if $\cA_\alpha = \Dec$. If we then perform the actual sorting step, 
each of the sequences will end up in the correct original order.

\subparagraph{Evaluation phase.}

After the sorting phase, the atomic pairs in the pattern are bundled into consecutive triples 
determined by clauses of $\Phi$. We show how to test whether a clause $K_j = (L_a \vee L_b \vee 
L_c)$ is satisfied where $L_\alpha \in \{x_\alpha, \neg x_\alpha\}$ for $\alpha \in \{a,b,c\}$ and 
$a < b < c$.

Suppose $X_a, X_b$ and $X_c$ are the three neighboring atomic pairs in $\cP$ that correspond to the 
three literals in $K_j$. In $\cT$, there are six neighboring atomic pairs $Y_a$, $Z_a$,  $Y_b$, 
$Z_b$,  $Y_c$, $Z_c$ in this precise order such that in any grid-preserving embedding, the pair 
$X_\alpha$ is mapped to either $Y_\alpha$ or $Z_\alpha$ for every $\alpha\in\{a,b,c\}$.

First, we claim that we can without loss of generality assume that in fact $K_j = (x_a \vee x_b \vee 
x_c)$. If that was not the case, we could use one layer of flip gadgets to reverse the order of 
$Y_\alpha$, $Z_\alpha$ for each $\alpha$ such that $L_\alpha = \neg x_\alpha$.

As in the sorting phase, we slightly abuse the notation and use the same letters to denote atomic 
pairs in different tiles so that any gadget with the same number of input and output atomic pairs 
carries the names through. We add one layer of gadgets, in particular
\begin{itemize}%[noitemsep]
  \item a choose gadget that branches $Z_b$ to $\overline{Z}_b$ and $\widetilde{Z}_b$, and
  \item pick gadgets to $Y_a, Z_a, Y_b, Y_c, Z_c, X_a, X_b$ and $X_c$.
\end{itemize}

We continue with adding two layers of flip gadgets, modifying the order of atomic pairs in the text 
as follows
\begin{align*}
  Y_a Z_a \flip{Y_b \overline{Z}_b} \widetilde{Z}_b \flip{Y_c Z_c} \rightarrow
  Y_a \flip{Z_a \overline{Z}_b} Y_b \flip{\widetilde{Z}_b Z_c} Y_c \rightarrow
  Y_a \overline{Z}_b Z_a Y_b Z_c \widetilde{Z}_b  Y_c,
\end{align*}
and two consecutive copies of flip pattern gadget to all $X_a$, $X_b$ and $X_c$.
This construction is done for each clause in parallel. Observe that it uses only constantly many 
layers of gadgets and thus uses only $O(1)$ $\cD$-entries of the path.

That concludes the construction of $\cP$ and $\cT$. Observe that each tile in both $\cP$ and $\cT$ 
contains $O(m)$ points. If the construction uses $L$ entries equal to $\cD$, then we need to start 
with a proper-turning path of length $(1/\eps)L = O(L)$ where $\eps$ is constant given by the 
$\cD$-rich property of $\cC$. However, the total amount of $\cD$-entries used by the construction 
depends on how many steps are needed for the sorting phase. If $\cD$ contains a monotone 
juxtaposition, then the total amount of $\cD$-entries used is $O(\log m)$ and thus the length of 
both $\pi$ and $\tau$ is $O(m \log m)$. Otherwise, if we sort using only gadgets, the total amount 
of $\cD$-entries used by the reduction is $O(m)$ and thus the length of both $\pi$ and $\tau$ is 
$O(m^2)$. This gives rise to the two different lower bounds for the runtime of an algorithm solving 
\CPPM{$\cC$} under ETH.

\subparagraph{Beyond grid-preserving embeddings.}
First, we modify both $\pi$ and $\tau$ so that any embedding that maps the image of $P_1$ to the 
image of $T_1$ must already be grid-preserving. To that end, take $\cP'$ to be the family of tiles 
obtained from $\cP$ by adding atomic pairs $A_1, A_2$ to the initial tile $P_1$ so that $A_1$ is 
to the left and below everything else in $P_1$ and $A_2$ is to the right and above everything else 
in~$P_1$. We then attach to both $A_1$, $A_2$ a chain of copy gadgets spreading all the way to the 
last tile used by $\cP$. We obtain $\cT'$ from $\cT$ in the same way by adding atomic pairs $B_1, 
B_2$ in $T_1$ and a chain of copy gadgets attached to each. We let $\pi'$ and $\tau'$ be the 
$\cF$-assemblies of $\cP'$ and~$\cT'$. 

Observe that in any embedding of $\pi'$ into $\tau'$ that maps the image of $P'_1$ to the image of 
$T'_1$, the image of $A_\alpha$ is mapped to the image of $B_\alpha$ for each $\alpha \in \{1,2\}$. 
The chain of copy gadgets attached to $A_\alpha$ then must map to the chain of gadgets attached to 
$B_\alpha$ and these chains force that the image of $P'_i$ maps to the image of $T'_i$ for 
every~$i$. Observe that $|\pi'| = O(|\pi|)$ and $|\tau'| = O(|\tau|)$.

Finally, we modify $\pi'$ and $\tau'$ to obtain permutations $\pi''$ and $\tau''$ such that any 
embedding of $\pi''$ into $\tau''$ can be translated to an embedding of $\pi'$ into $\tau'$ that 
maps $P'_1$ to $T'_1$ and vice versa. Let $r$ be the lowest point in $P'_1$ (i.e. the lower point of 
$A_1$) and let $q$ be the topmost point of $P'_1$ (i.e. the upper point of $A_2$). Similarly, let 
$s$ be the lowest point in $T'_1$ and $t$ be the topmost point of $T'_1$ in $\tau'$. The family of 
tiles $\cP''$ is obtained from $\cP'$ by inflating both $r$ and $q$ with an increasing sequence of 
length $|\tau'| + 1$ and similarly, the family $\cT''$ is obtained from $\cT''$ by inflating both 
$s$ and $t$ with an increasing sequence of length $|\tau'| + 1$. We call the points obtained by 
inflating $r$ and $q$ \emph{lower anchors} and the ones obtained by inflating $s$ and $t$  
\emph{upper anchors}. We let $\pi''$ and $\tau''$ be the $\cF$-assemblies of $\cP''$ and $\cT''$.  
Observe that these modifications did not change the asymptotic size of the input as $|\pi''| = 
O(|\tau'|) = O(|\tau|)$ and $|\tau''| = O(|\tau'|)= O(|\tau|)$.

\subsection{Correctness}

The construction described so far shows how to construct, from a 3-SAT formula $\Phi$ with $n$ 
variables and $m$ clauses, a pair of permutations $\pi''$ and $\tau''$. It follows from the 
construction that both $\pi''$ and $\tau''$ belong to the class $\cC$ from 
Theorem~\ref{thm-hardness}. Additionally, if $\cD$ contains a monotone juxtaposition, we may 
implement the sorting phase via juxtapositions, and the size of $\pi$ and $\tau$ is bounded by 
$O(m\log m)$; otherwise we implement sorting via gadgets and the size of $\pi$ and $\tau$ is bounded 
by $O(m^2)$. 

The ETH, together with the Sparsification Lemma of Impagliazzo, Paturi and Zane~\cite{IPZ}, 
implies that no algorithm may solve 3-SAT in time $2^{o(n+m)}$. Thus, to prove 
Theorem~\ref{thm-hardness}, it remains to prove the correctness of the reduction, i.e., to show that 
$\Phi$ is satisfiable if and only if $\pi$ is contained in~$\tau$.

\subparagraph{The ``only if'' part}
Let $\Phi$ be a satisfiable formula and fix arbitrary satisfying assignment represented by a 
function $\rho\colon [n] \to \{T,F\}$, where $\rho(k) = T$ if and only if the variable $x_k$ is set 
to true in the chosen assignment.

We map the image of $P''_1$ to the image of $T''_1$. In the assignment phase, we map the pair 
$X^1_{k,1}$ to $Y^1_{k,1}$ if $\rho(k) = T$, otherwise we map it to $Z^1_{k,1}$. The embedding of 
the multiplication phase is then uniquely determined by the properties of the gadgets. In particular 
at the end of the multiplication phase, the pair $X^\ell_{k,j}$ for every $j$ is mapped to 
$Y^\ell_{k,j}$ if $\rho(k) = T$ and to $Z^\ell_{k,j}$ otherwise.

The mapping is then straightforwardly extended through the sorting phase. We just have to be careful 
when swapping two neighboring pairs to pick correctly between $\overline{Y}_\alpha$ and 
$\widetilde{Y}_\alpha$ (or $\overline{Z}_\alpha$ and $\widetilde{Z}_\alpha$) such that the swap 
itself is possible.

Recall that at the end of the sorting phase, we have for each clause $K_j = (L_a \vee L_b \vee L_c)$ 
a consecutive block of three pairs $X_a, X_b, X_c$ in a pattern tile and a consecutive block of six 
pairs $Y_a$, $Z_a$,  $Y_b$, $Z_b$,  $Y_c$, $Z_c$ in a text tile such that $X_\alpha$ is mapped to 
$Y_\alpha$ if $\rho(\alpha) = T$ for every $\alpha = \{a,b,c\}$. We also showed that by attaching an 
extra layer of flip gadgets and suitably renaming the atomic pairs, we can assume that $K_j$ is, in 
fact, $(x_a \vee x_b \vee x_c)$. If $\rho(b) = T$ then the mapping is uniquely determined and 
possible regardless of the values $\rho(a)$ and $\rho(b)$. On the other hand, if $\rho(b) = F$ we 
need to select where to map the pair $X_b$ in the choose gadget that branches $Z_b$ to 
$\overline{Z}_b$ and $\widetilde{Z}_b$. We pick the pair $\overline{Z}_b$ if $\rho(a) = T$ and the 
pair $\widetilde{Z}_b$ if $\rho(c)= T$. Note that at least one of the two cases must happen; 
otherwise, the clause $K_j$ would not be satisfied. It is easy to see that in both cases, the 
relative order of the pairs in the text does not change, and we defined a valid embedding. 

\subparagraph{The ``if'' part}
Let $\phi$ be an embedding of $\pi''$ into $\tau''$. The total length of the anchors in both  
$\pi''$ and $\tau''$ is $2|\tau'| + 2$. Therefore, at least $|\tau'| + 2$ points of the anchors in 
$\pi''$ must be mapped to the anchors in $\tau''$ as there are only $|\tau'|$ remaining points. In 
particular, there is at least one point in each anchor of $\pi''$ that maps to anchors in $\tau''$. 
Moreover, this implies that there is a point $q'$ in the lower anchor of $\pi''$ mapped to the point 
$r'$ in the lower anchor of $\tau''$ and there is a point $s'$ in the upper anchor of $\pi''$ mapped 
to the point $t'$ in the upper anchor of $\tau''$.

We claim that no point of $P''_1$ that does not belong to the anchors can be mapped to the anchors 
in $\tau''$. This holds since there are copy gadgets attached to both $A_1$ and $A_2$ that 
vertically separate the anchors from the rest of the tile $P''_1$. Furthermore, the chains of copy 
gadgets attached to $A_1$ and $A_2$ then force the rest of the embedding to be grid-preserving, and 
thus it straightforwardly translates to a grid-preserving embedding of $\pi$ into~$\tau$.

Using the grid-preserving embedding, we now define a satisfying assignment $\rho\colon [n] \to 
\{T,F\}$. We set $\rho(k) = T$ if the pair $X^1_{k,1}$ is mapped to $Y^1_{k,1}$ and we set $\rho(k) 
= F$ if it is mapped to $Z^1_{k,1}$. This property is clearly maintained throughout the 
multiplication and sorting phases due to the properties of gadgets.

At the beginning of the evaluation phase, we thus have for each clause $K_j = (L_a \vee L_b \vee 
L_c)$ a consecutive block of three pairs $X_a, X_b, X_c$ in a pattern tile and a consecutive block 
of six pairs $Y_a$, $Z_a$,  $Y_b$, $Z_b$,  $Y_c$, $Z_c$ in a text tile such that $X_\alpha$ is 
mapped to $Y_\alpha$ if and only if $\rho(\alpha) = T$ for every $\alpha = \{a,b,c\}$. As before, we 
assume that $K_j = (x_a \vee x_b \vee x_c)$. It remains to argue that it cannot happen that the pair 
$X_\alpha$ is mapped to $Z_\alpha$ for every $\alpha \in \{a,b,c\}$. If that was the case, the pairs 
$X_a$, $X_b$, $X_c$ in the very last tile would be mapped either to the pairs $Z_a$, 
$\widetilde{Z}_b$, $Z_c$ or to the pairs $Z_a$, $\overline{Z}_b$, $Z_c$ . However in both these 
cases, the relative order of these pairs differs from the order of $X_a$, $X_b$ and $X_c$ in the 
pattern and we arrive at contradiction. Thus, every clause is satisfied with the assignment given 
by~$\rho$.

This completes the proof of Theorem~\ref{thm-hardness}.

\subsection{Consequences}

In the rest of this section, we focus on presenting examples of classes that satisfy the technical 
``rich path'' property, which is the backbone of all our hardness arguments.

\begin{proposition}\label{pro-cyclehard}
Let $\cD$ be a non-monotone-griddable class that is sum-closed or skew-closed.
If $\cM$ is a gridding matrix whose cell graph $G_\cM$ contains a proper-turning cycle with at least 
one entry equal to $\cD$, then $\Grid(\cM)$ has the computable $\cD$-rich path property.
\end{proposition}

\begin{proof}
% We suppose that $\cD$ is either sum-closed or skew-closed.
% In the first case this is already guaranteed, and in the second case, we simply take instead of 
% $\cD$ just the class $ \oplus21$ or $\ominus 12$ contained in $\cD$ due to 
% Theorem~\ref{thm-monotone-griddable}.
% Our goal is to show that $\Grid(\cM)$ has the computable $\cD$-rich path property; the claim then 
% follows from Theorem~\ref{thm-hardness}.
We note that the proof closely follows a proof of a similar claim for monotone grid classes by 
Jelínek et al.~\cite[Lemma 3.5]{Jelinek2020}.

We may assume, without loss of generality, that the cell graph of $\cM$ consists of a single cycle, 
that it contains a unique entry equal to $\cD$, and that all the remaining nonempty entries are 
equal to $\Inc$ or to~$\Dec$. This is because each infinite permutation class contains either $\Inc$ 
or $\Dec$ as a subclass, and replacing an entry of $\cM$ by its infinite subclass can only change   
$\Grid(\cM)$ into its subclass. If we can establish the $\cD$-rich path property for the subclass, 
then it also holds for the class $\Grid(\cM)$ itself.

We may also assume that $\cD$ is sum-closed, since the skew-closed case is symmetric. In 
particular, $\cD$ contains $\oplus21$ as a subclass. 

\begin{figure}
\includegraphics[width=\textwidth]{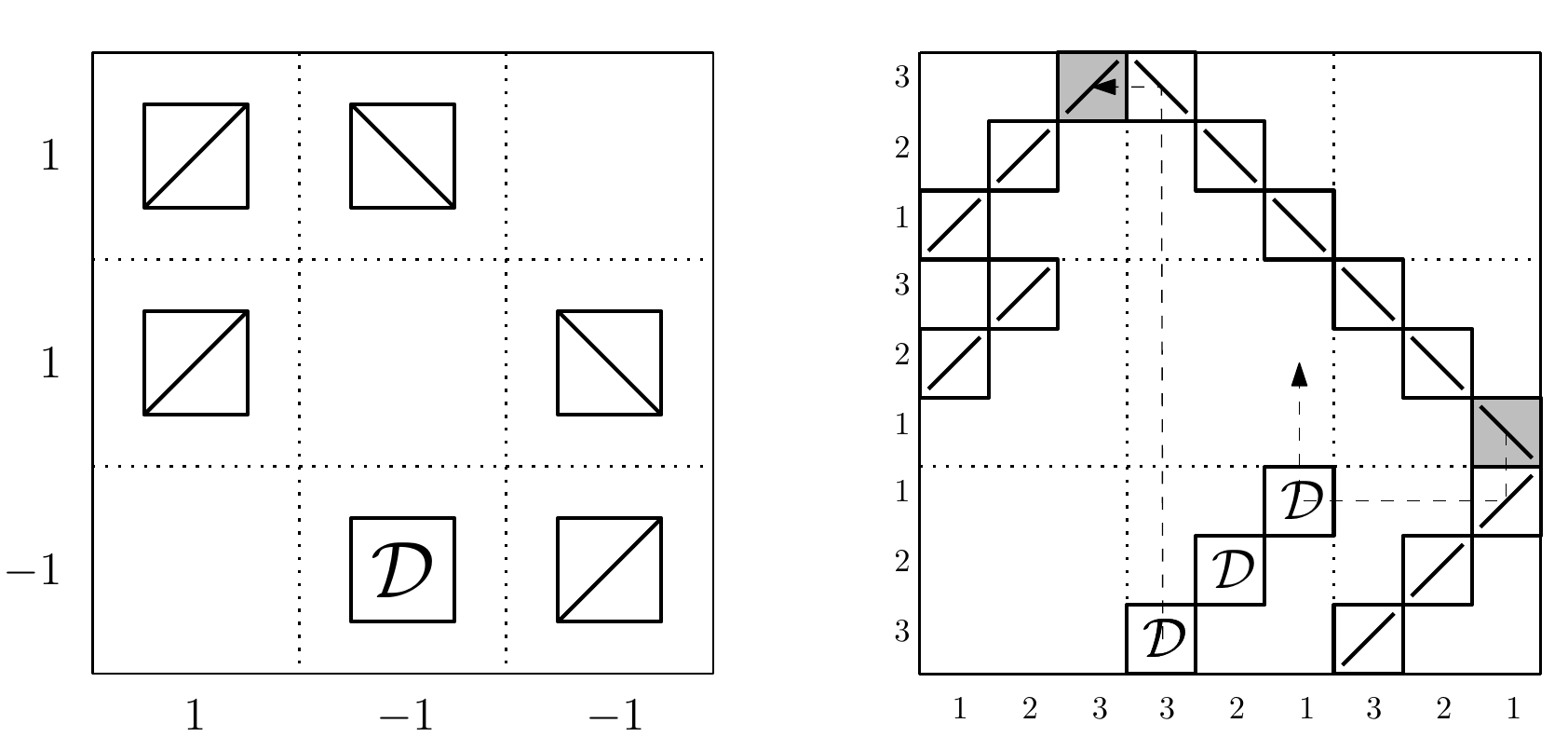}
\caption{Illustration of the proof of Proposition~\ref{pro-cyclehard}. Left: a gridding matrix 
$\cM$ whose cell graph is a cycle with a single entry equal to~$\cD$. The numbers along the bottom 
and the left edge form an orientation that maps each entry to a sum-closed class. Right: a gridding 
matrix whose grid class is contained in $\Grid(\cM)$ and whose cell graph is a path. The endvertices 
of the path are highlighted in gray, and the first and last few steps of the path are shown as dashed 
lines. The numbers along the edges are the labels forming the characteristic of each 
entry.}\label{fig-gridpath}
\end{figure}

Let $L$ be a given integer. We show how to obtain a grid subclass of $\Grid(\cM)$ whose cell graph is 
a proper-turning path of length at least $L$ that contains a constant fraction of $\cD$-entries. 
Refer to Figure~\ref{fig-gridpath}.
Suppose $\cN$ is a $k \times \ell$ gridding matrix whose every entry is either sum-closed or skew-closed. The \emph{refinement $\cN^{\times q}$} of $\cN$ is the $qk \times q\ell$ matrix obtained from $\cN$ by replacing the entry $\cN_{i,j}$ with
\begin{itemize}%[noitemsep]
\item a $q\times q$ diagonal matrix with all the non-empty entries equal to $\cN_{i,j}$ if $\cN_{i,j}$ is sum-closed,
\item a $q\times q$ anti-diagonal matrix with all the non-empty entries equal to $\cN_{i,j}$ if $\cN_{i,j}$ is skew-closed.
\end{itemize}
It is easy to see that $\Grid(\cN^{\times q})$ is a subclass of $\Grid(\cN)$. We call the submatrix  
of $\cN^{\times q}$ formed by the entries $\cN^{\times q}_{a,b}$ for $q\cdot i < a \le (q+1) \cdot i$ 
and $q\cdot j < b \le (q+1) \cdot j$ the \emph{$(i,j)$-block} of $\cN^{\times q}$. 

Importantly, it follows from the work of Albert et al.~\cite[Proposition 4.1]{Albert2013} that for every monotone gridding matrix $\cN$, there exists a consistent orientation of the refinement $\cN^{\times 2}$.
Translating it to our setting, we can assume that there is a $k \times \ell$ orientation $\cF$ such 
that the image of $\cM_{i,j}$ under $\cF$ is sum-closed for every $i \in [k]$ and $j \in [\ell]$. 
If that is not the case for $\cM$, we simply start with $\cM^{\times 2}$ instead.

Given such an orientation $\cF = (f_c, f_r)$, we label the rows and columns of the refinement 
$\cM^{\times L}$ using the set $[L]$.
The $L$-tuple of columns created from the $i$-th column of $\cM$ is labeled in the increasing order 
from left to right if $f_c(i)$ is positive and right to left otherwise.
Similarly, the $L$-tuple of rows created from the $j$-th row of $\cM$ is labeled in the increasing 
order from bottom to top if $f_r(j)$ is positive and top to bottom otherwise.
The \emph{characteristic of an entry} in $\cM^{\times L}$ is the pair of labels given to its column and row.
Observe that each non-empty entry in $\cM^{\times L}$ has a characteristic of the form $(s,s)$ for 
some $s \in [L]$ by the choice of orientation.
Therefore, $G_{\cM^{\times L}}$ consists exactly of $L$ connected components, each corresponding to a 
copy of~$\cM$.

We pick an arbitrary non-empty monotone entry $\cM_{i,j}$ of $\cM$ and obtain a matrix $\cM_L$ by 
replacing the $(i,j)$-block in $\cM^{\times q}$ with the $q \times q$ matrix whose only non-empty 
entries are the ones with characteristic $(s,s+1)$ for all $s \in [L-1]$ and they are all equal to 
$\cM_{i,j}$. $\Grid(\cM_L)$ is a subclass of $\Grid(\cM)$ since the modified $(i,j)$-block 
corresponds to shifting the original (anti-)diagonal matrix by one row either up or down, depending 
on the orientation of the $j$-th row of~$\cM$.

Observe that we connected all the $L$ copies of $\cM$ into a single long path.
Moreover, the path contains $L-1$ entries in the $(i,j)$-block and $L$ entries in every other non-empty block. Therefore, a constant fraction of its entries belong to the $(a,b)$-block such that $\cM_{a,b} = \cD$ and thus are equal to $\cD$.
It is easy to see that the described procedure is constructive and can easily be implemented to run in polynomial time.
Therefore, $\Grid(\cM)$ indeed has the computable $\cD$-rich path property.
\end{proof}

Combining Proposition~\ref{pro-cyclehard} with Theorem~\ref{thm-hardness}, we get the following 
corollary. Note that in the corollary, if $\cD$ fails to be sum-closed or skew-closed, we may 
simply replace it with $\oplus21$ or $\ominus12$, since at least one of these two classes is its 
subclass by Theorem~\ref{thm-monotone-griddable}.

\begin{corollary}
\label{cor:cycle-hard}
Let $\cD$ be a non-monotone-griddable class. If $\cM$ is a gridding matrix whose cell graph contains 
a proper-turning cycle with one entry equal to $\cD$, then \CPPM{$\Grid(\cM)$} is NP-complete. 
Moreover, unless ETH fails, there can be no algorithm for \CPPM{$\Grid(\cM)$} running
\begin{itemize}
	\item in time $2^{o(n /\log{n})}$ if $\cD$ additionally contains any monotone juxtaposition and 
is either sum-closed or skew-closed,
	\item in time $2^{o(\sqrt{n})}$ otherwise.
\end{itemize} 
\end{corollary}

Three symmetry types of patterns of length 4 can be tackled with a special type of grid classes. The \emph{$k$-step increasing $(\cC, \cD)$-staircase}, denoted by $\St_k(\cC, \cD)$ is a grid class $\Grid(\cM)$ of a $k \times (k+1)$ gridding matrix $\cM$ such that the only non-empty entries in $\cM$ are $\cM_{i,i} = \cC$ and $\cM_{i, i+1} = \cD$ for every $i \in [k]$. In other words, the entries on the main diagonal are equal to $\cC$ and the entries of the adjacent lower diagonal are equal to $\cD$. The \emph{increasing $(\cC, \cD)$-staircase}, denoted by $\St(\cC, \cD)$, is the union of $\St_k(\cC, \cD)$ over all $k \in \mathbb{N}$.

Observe that if $\cC$ and $\cD$ are two infinite classes and one of them contains $\oplus21$ or $\ominus 12$ then Theorem~\ref{thm-hardness} applies and \CPPM{$\St(\cC, \cD)$} is NP-complete. Furthermore, if it also contains a monotone juxtaposition as a subclass, then the almost linear lower bound under ETH follows.
We proceed to show that three symmetry types of classes avoiding a pattern of length 4 actually contain such a staircase subclass.

\begin{proposition}
\label{prop:staircases}
For any sum-indecomposable permutation $\sigma$, $\St(\Inc, \Av(\sigma))$ is a subclass of $\Av(1 
\ominus \sigma)$.
\end{proposition}
\begin{proof}
Suppose for a contradiction that $\sigma' = 1 \ominus \sigma$ belongs to $\St(\Inc, \Av(\sigma))$. In particular it belongs to $\St_k(\Inc, \Av(\sigma))$ for some $k$ and there is a witnessing gridding. If the first element is not mapped to one of the $\Inc$-entries on the upper diagonal, then the whole $\sigma'$ must lie in a single $\Av(\sigma)$-entry on the lower diagonal, which is clearly not possible. Therefore, the first element must be mapped to one of the $\Inc$-entries. Notice that the rest of $\sigma'$ cannot be mapped to any of the $\Inc$-entries as it lies below and to the right of the first element. However, it cannot lie in more than one $\Av(\sigma)$-entry; otherwise, we could express $\sigma$ as a direct sum of two shorter permutations. Hence, there must be an occurrence of $\sigma$ in an $\Av(\sigma)$-entry which is clearly a contradiction.
\end{proof}

A direct consequence of Proposition~\ref{prop:staircases} is that taking $\sigma$ to be $321$, $312$ or $231$, we see that $\St(\Inc,\Av(321))\subseteq \Av(4321)$, $\St(\Inc,\Av(231))\subseteq\Av(4231)$ and $\St(\Inc,\Av(312))\subseteq \Av(4312)$. Note that the first inclusion is rather trivial and the latter two have been previously observed by Berendsohn~\cite{BerendsohnMs}.

We may easily observe that for any pattern $\sigma$ of size 3, the class $\Av(\sigma)$ contains the Fibonacci class or its reversal, as well as a monotone juxtaposition. Combining Proposition~\ref{prop:staircases} with Theorem~\ref{thm-hardness} yields the following consequence.

\begin{corollary}\label{cor-size4}
For any permutation $\sigma$ that contains a pattern symmetric to $4321$, to $4231$, or to $4312$, the problem \CPPM{$\Av(\sigma)$} is NP-complete, and unless ETH fails, it cannot be solved in time $2^{o(n/\log n)}$.
\end{corollary}

We verified by computer that there are only five symmetry types of patterns of length $5$ that do 
not contain any of $4321$, $4213$, $4312$ or their symmetries --- represented by $14523$, $24513$, 
$32154$, $42513$ and $41352$. Of these five, four can be handled by Corollary~\ref{cor:cycle-hard} 
since they contain a specific type of cyclic grid classes, as we now show.

\begin{proposition}
\label{prop:rich-2by2}
The class $\Av(\sigma)$ contains the class $\Grid(\cM)$ for the gridding matrix $\cM = \begin{psmallmatrix} \Dec & \Inc \\ \Av(\pi) & \Dec \end{psmallmatrix}$ whenever
\begin{itemize}%[noitemsep]
\item $\pi = 132$ and $\sigma = 14523$, or
\item $\pi = 231$ and $\sigma = 24513$, or
\item $\pi = 321$ and $\sigma \in \{32154, 42513\}$.
\end{itemize}
\end{proposition}
\begin{proof}
Suppose that $\sigma$ and $\pi$ are one of the listed cases. Observe that $\Grid(\cM)$ is a subclass of $\Av(\sigma)$ if and only if $\sigma$ is not in $\Grid(\cM)$. For contradiction, suppose that
the class $\Grid(\cM)$ contains $\sigma$. Therefore, there exists a witnessing $\cM$-gridding $1 = c_1 \le c_2 \le c_3 = 6$ and $1 = r_1 \le r_2 \le r_3 = 6$ of~$\sigma$. 

Let us consider the four choices of $\sigma$ separately, starting with $\sigma=14523$: if $c_2\le 3$ and $r_2\le 3$, the cell $(2,2)$ of the gridding contains the pattern $21$, if $c_2\le 4$ and $r_2\ge 4$, the cell $(2,1)$ contains $12$, if $c_2\ge 4$ and $r_2\le 4$, the cell $(1,2)$ contains $12$, and if $c_2\ge 5$ and $r_2\ge 5$, the cell $(1,1)$ contains $132$. In all cases we get a contradiction with the properties of the $\cM$-gridding. The same argument applies to $\sigma=14513$, except in the last case we use the pattern $231$ instead of $132$.

For $\sigma=32154$, the four cases to consider are $c_2\le 4 \land r_2\le 4$, $c_2\ge 5\land r_2\le 
3$, $c_2\le 3\land r_2\ge 5$, and $c_2\ge 4\land r_2\ge 4$, in each case getting contradiction in a 
different cell of the gridding. For $\sigma=42513$, the analogous argument distinguishes the cases 
$c_2\le 3 \land r_2\le 3$, $c_2\ge 4\land r_2\le 4$, $c_2\le 4\land r_2\ge 4$, and $c_2\ge 5\land 
r_2\ge 5$.
%We claim that $c_2 \ge 5$. It is easy to check that $c_2=4$ forces such a choice of $r_2$ that the $(1,2)$-cell of $\sigma$ contains $12$ when $\sigma$ is one of $14523$, $24513$, $42513$ and in the case $\sigma = 32145$ it forces the $(1,1)$-cell to contain $321$. Either way, that is not possible. Analogously, it can be checked that $r_2 \ge 5$ because setting $r_2 = 4$ forces such a choice of $c_2$ that the $(2,1)$-cell of $\sigma$ contains $12$ when $\sigma$ is one of $14523$, $24513$, $42513$ and in the case $\sigma = 32145$ it forces the $(1,1)$-cell to contain $321$.
%
%Therefore, $c_2 \ge 5$ and $r_2 \ge 5$ and at least three points belong to the $(1,1)$-cell of the gridding. In all the cases, this is precisely the permutation $\pi$ that cannot occur there.\mo{The original proof was wrong. This, on the other hand, is probably too much of `left as an exercise'.}
\end{proof}

It is easy to see that every $\sigma$ of length at least 6 contains a pattern of size 5 which is not 
symmetric to $41352$. Therefore, \CPPM{$\Av(\sigma)$} is NP-complete for all permutations $\sigma$ of length 
at least 4 except for one symmetry type of length 5 and for four out of seven symmetry types of 
length~4. As \CPPM{$\Av(\sigma)$} is polynomial-time solvable for any $\sigma$ of length at most 3, 
these are, in fact, the only cases left unsolved.

\begin{corollary}
If $\sigma$ is a permutation of length at least 4 that is not in symmetric to any of $3412, 3142, 4213, 4123$ or $41352$, then \CPPM{$\Av(\sigma)$} is NP-complete, and unless ETH fails, it cannot be solved in time $2^{o(n /\log{n})}$.
\end{corollary}

To conclude this section, we remark that the suitable grid subclasses were discovered via computer experiments facilitated by the Permuta library~\cite{permuta}.

\section{Polynomial-time algorithm}

We say that a permutation $\pi$ is \emph{$t$-monotone} if there is a partition $\Pi\! =\! (S_1, 
\dotsc, S_t)$ of $S_\pi$ such that $S_i$ is a monotone point set for each $i \in [t]$. The partition 
$\Pi$ is called a \emph{$t$-monotone partition}.

Given a $t$-monotone partition $\Pi = (S_1, \dots, S_t)$ of a permutation $\pi$ and a $t$-monotone partition $\Sigma = (S'_1, \dots, S'_t)$ of $\tau$, an embedding $\phi$ of $\pi$ into $\tau$ is a \emph{$(\Pi, \Sigma)$-embedding} if $\phi(S_i) \subseteq S'_i$ for every $i\in [t]$. Guillemot and Marx~\cite{Guillemot2014} showed that if we fix a $t$-monotone partitions of both $\pi$ and $\tau$, the problem of finding a $(\Pi, \Sigma)$-embedding is polynomial-time solvable.

\begin{proposition}[Guillemot and Marx~\cite{Guillemot2014}]
\label{pro-monotone}
Given a permutation $\pi$ of length $m$ with a $t$-monotone partition $\Pi$ and a permutation $\tau$ of length $n$ with a $t$-monotone partition $\Sigma$, we can decide if there is a $(\Pi, \Sigma)$-embedding of $\pi$ into $\tau$ in time $O(m^2n^2)$.
\end{proposition}

We can combine this result with the fact that there is only a bounded number of ways how to grid a permutation, and obtain the following counterpart to Corollary~\ref{cor:cycle-hard}.

\begin{theorem}\label{thm-monotone}
\CPPM{$\cC$} is polynomial-time solvable for any monotone-griddable class $\cC$.
\end{theorem}
\begin{proof}
Let $\cM$ be a $k \times \ell$ monotone gridding matrix such that $\Grid(\cM)$ contains the class $\cC$. We have to decide whether $\pi$ is contained in $\tau$ for two given permutations $\pi$ of length $m$ and $\tau$ of length $n$, both belonging to the class $\cC$.

First, we find an $\cM$-gridding of $\tau$. We enumerate all possible $k \times \ell$ griddings and 
for each, we test if it is a valid $\cM$-gridding. Observe that there are in total $O(n^{k + \ell - 
2})$ such griddings since they are determined by two sequences of values from the set $[n]$, one of 
length $k-1$ and the other of length $\ell-1$. Moreover, it is straightforward to test  in time 
$O(n^2)$ whether a given $k \times \ell$ gridding is in fact an $\cM$-gridding.  Note that we are 
guaranteed to find an $\cM$-gridding as $\tau$ belongs to $\cC \subseteq \Grid(\cM)$. We set 
$\Sigma$ to be the $(k \cdot \ell)$-monotone partition of $\tau$ into the monotone sequences given 
by the individual cells of the gridding.

In the second step, we enumerate all possible $\cM$-griddings of $\pi$. As with $\tau$, we enumerate all possible $O(m^{k + \ell - 2})$ $k \times \ell$ griddings of $\pi$ and check for each gridding whether it is actually an $\cM$-gridding in time $O(m^2)$. For each $\cM$-gridding found, we let $\Pi$ be the $(k \cdot \ell)$-monotone partition of $\pi$ given by the gridding, and we apply Proposition~\ref{pro-monotone} to test whether there is a $(\Pi, \Sigma)$-embedding in time $O(m^2 n^2)$.

If there is an embedding $\phi$ of $\pi$ into $\tau$, there is a $(k \cdot \ell)$-monotone partition $\Sigma'$ of $\pi$ such that $\phi$ is a $(\Pi, \Sigma')$-embedding. Therefore, the algorithm correctly solves \CPPM{$\cC$} in time $O(n^{k+\ell} + m^{k+\ell}n^2)$ --- polynomial in $n, m$.
\end{proof}

Notice that if $\cM$ is a gridding matrix whose every entry is monotone griddable, or equivalently no 
entry contains the Fibonacci class or its reverse as a subclass, then the class $\Grid(\cM)$ is 
monotone griddable as well. It follows that for such $\cM$, the \CPPM{$\Grid(\cM)$} problem is 
polynomial-time solvable. We also note that recent results on \PPPM{$\cC$}~\cite{Jelinek2020} imply 
that if a gridding matrix $\cM$ has an acyclic cell graph, and if every nonempty cell is either 
monotone or symmetric to a Fibonacci class, then \PPPM{$\Grid(\cM)$}, and therefore also 
\CPPM{$\Grid(\cM)$}, is polynomial-time solvable as well. These two tractability results contrast 
with our Corollary~\ref{cor:cycle-hard}, which shows that for any gridding matrix $\cM$ whose cell 
graph is a cycle, and whose nonempty cells are all monotone except for one Fibonacci cell, 
\CPPM{$\Grid(\cM)$} is already NP-hard.

\section{Open problems}

We have presented a hardness reduction which allowed us to show that the \CPPM{$\Av(\sigma)$} 
problem is NP-complete for every permutation $\sigma$ of size at least 6, as well as for most 
shorter choices of~$\sigma$. Nevertheless, for several symmetry types of $\sigma$, the complexity 
of \CPPM{$\Av(\sigma)$} remains open. We collect all the remaining unresolved cases as our first 
open problem.

\begin{problem}
What is the complexity of \CPPM{$\Av(\sigma)$}, when $\sigma$ is a permutation from the set 
$\{3412, 3142, 4213, 4123, 41352\}$?
\end{problem}

Our hardness results are accompanied by time complexity lower bounds based on the ETH. 
Specifically, for our NP-hard cases, we show that under ETH, no algorithm may solve \CPPM{$\cC$} in time 
$2^{o(\sqrt{n})}$. The lower bound can be improved to $2^{o(n/\log n)}$ under additional 
assumptions about~$\cC$. This opens the possibility of a more refined complexity hierarchy within 
the NP-hard cases of \CPPM{$\cC$}. In particular, we may ask for which $\cC$ can 
\CPPM{$\cC$} be solved in subexponential time.

\begin{problem}
Which cases of \CPPM{$\cC$} can be solved in time $2^{O(n^{1-\eps})}$? Can the general PPM problem 
be solved in time $2^{o(n)}$?
\end{problem}

\bibliography{bibliography}
\end{document}